\documentclass[a4paper, 11 pt]{article}
\usepackage{a4wide, amsmath, amssymb, mathtools, yfonts}
\usepackage{comment}
\usepackage{tikz}
\usepackage{tikz-cd}
\usepackage[all]{xy}
\usepackage[utf8]{inputenc}
\usepackage{amsthm, mathrsfs}
\usepackage[english]{babel}
\usepackage{hyperref}
\usepackage{authblk}
\usepackage[OT2,T1]{fontenc}
\usepackage[short-journals,initials,msc-links,backrefs]{amsrefs}

\DeclareSymbolFont{cyrletters}{OT2}{wncyr}{m}{n}
\DeclareMathSymbol{\Sha}{\mathalpha}{cyrletters}{"58}

\numberwithin{equation}{section}
\newtheorem{lemma}{Lemma}[section]
\newtheorem{theorem}[lemma]{Theorem}
\newtheorem{proposition}[lemma]{Proposition}
\newtheorem{corollary}[lemma]{Corollary}

\newtheorem{remark}[lemma]{Remark}

\theoremstyle{definition}
\newtheorem{mydef}[lemma]{Definition}

\newtheorem{example}[lemma]{Example}

\newcommand{\Z}{\mathbb{Z}}
\newcommand{\Q}{\mathbb{Q}}

\newcommand{\R}{\mathbb{R}}

\newcommand\Vol{\mathrm{Vol}}
\newcommand\Gal{\mathrm{Gal}}

\newcommand{\res}{\textup{res}}

\usepackage{epigraph}
\allowdisplaybreaks

\title{\vspace{-\baselineskip}\sffamily\bfseries Constructing Jacobians of rank $1$}
\author[1]{Peter Koymans\thanks{Mathematisch Instituut, Universiteit Utrecht, Postbus 80.010, 3508 TA Utrecht, The Netherlands, p.h.koymans@uu.nl}}
\author[2]{Adam Morgan\thanks{Gonville and Caius College, Trinity Street, Cambridge, CB2 1TA, United Kingdom, ajm269@cam.ac.uk}}%Centre for Mathematical Sciences, Wilberforce Road, Cambridge CB3 0WA, United Kingdom, ajm269@cam.ac.uk}}
\affil[1]{Utrecht University}
\affil[2]{University of Cambridge}
\date{\today}

\begin{document}
\maketitle

\begin{abstract}
Let $K$ be a number field, let $g \geq 1$ be an integer and let $f(x) = (x - a_1) \cdots (x - a_{2g + 1}) \in O_K[x]$ be a polynomial that splits into $2g + 1$ distinct linear factors. Write $C$ for the hyperelliptic curve given by $C: y^2 = f(x)$ and write $J = \mathrm{Jac}(C)$ for its Jacobian. Under mild technical assumptions on $f$ that are satisfied almost always, we prove that there exists some $d \in K^\times$ such that the quadratic twist $J^d$ has rank exactly equal to $1$. As a consequence, we deduce that for any  positive integer $g$, there exists an absolutely simple abelian variety over $K$ with dimension equal to $g$ and rank equal to $1$. 
\end{abstract}

\section{Introduction}
\subsection{Background and results}
An important but exceedingly difficult question in number theory is to determine the set
$$
\mathcal{R} := \{r : r = \mathrm{rk} \, E(\Q) \text{ for infinitely many non-isomorphic elliptic curves } E/\Q\}.
$$
Despite many computations and conjectures \cites{BW, Del1, Del2, DD, DW, Goldfeld, KatzSarnak, PPVW, RS1, RS2, Watkins1, Watkins2, Watkins3}, very little has been rigorously proven about the set $\mathcal{R}$. In this paper, for a number field $K$, we will be interested in understanding the more general set
$$
\mathcal{R}(K, g) := \left\{r : \begin{array}{l} r = \mathrm{rk} \, \mathrm{Jac}(C)(K) \text{ for infinitely many non-isomorphic} \\ \text{hyperelliptic curves } C/K \text{ of genus } g\end{array}\right\},
$$
so $\mathcal{R} = \mathcal{R}(\Q, 1)$. The main goal of this paper is to prove the following theorem, which implies that $1 \in \mathcal{R}(K, g)$ for all number fields $K$ and all integers $g \geq 1$.

\begin{theorem}
\label{tFirstResult}
Let $K$ be a number field and let $g \geq 1$ be an integer. Then there are infinitely many non-isomorphic hyperelliptic curves $C$ of the shape
$$
y^2 = (x - a_1) \cdots (x - a_{2g + 1})
$$
with distinct $a_1, \dots, a_{2g + 1} \in K$ such that its Jacobian $J = \mathrm{Jac}(C)$ satisfies $\mathrm{rk} \, J(K) = 1$. Moreover, the Jacobian $J$ may be taken to be absolutely simple. 
\end{theorem}

Theorem \ref{tFirstResult} positively resolves a question of Gajovic--Park \cite{GP}*{Question 4.1}. 

We will now state a more technical, but also significantly more powerful result, that easily implies Theorem \ref{tFirstResult}. To do so, we need some notation.

%\begin{mydef} 
%\label{def:places_in_aux_twist}
%We say that $J$ is $B$-generic if there are distinct finite places $w_1, \dots, w_{4g^2 + 2g - 1}$ and $w_1', \dots, w_{2g}'$ such that
%\begin{itemize}
%\item for all $1 \leq i \leq 2g$ and all $1 \leq j \leq 2g$, the place $w_{(i - 1) 2g + j}$ is a multiplicative prime of type $\{j, 2g + 1\}$ (see Definition \ref{dMultiplicative} for the definition of multiplicative prime);
%\item for all $1 \leq i \leq 2g - 1$, the place $w_{4g^2 + i}$ is a multiplicative prime of type $\{i, i + 1\}$;
%\item for all $1 \leq i \leq 2g$, the place $w_i'$ is a multiplicative prime of type $\{i, 2g + 1\}$;
%\item we have $|O_K/w_i| \geq B$ for all $i$.
% ADAM: Do we use this for $w_j'$? PETER: I think not, removed now.
%\end{itemize}
%\end{mydef}

\begin{mydef} 
\label{def:places_in_aux_twist}
We say that $J$ is $B$-generic if there are distinct finite places $w_1, \dots, w_{4g^2 + 2g - 1}$ and $w_1', \dots, w_{2g}'$ such that
\begin{itemize}
\item for all $1 \leq i \leq 2g$ and all $1 \leq j \leq 2g$, the place $w_{(i - 1) 2g + j}$ is multiplicative of type $\{j, 2g + 1\}$ (see Definition \ref{dMultiplicative} for the definition of multiplicative place);
\item for all $1 \leq i \leq 2g - 1$, the place $w_{4g^2 + i}$ is multiplicative of type $\{i, i + 1\}$;
\item for all $1 \leq i \leq 2g$, the place $w_i'$ is multiplicative of type $\{i, 2g + 1\}$;
\item we have $|O_K/w_i| \geq B$ for all $i$.
\end{itemize}
\end{mydef}

\begin{theorem}
\label{tMain}
Let $g \geq 1$ be an integer. Then there exists a constant $B > 0$ such that the following holds. Let $K$ be a number field and let $C$ be a hyperelliptic curve of the shape
$$
y^2 = (x - a_1) \cdots (x - a_{2g + 1})
$$
with distinct $a_1, \dots, a_{2g + 1} \in O_K$. Assume that $J = \mathrm{Jac}(C)$ is $B$-generic. Then $J$ is absolutely simple. Moreover, there exist infinitely many elements $t \in K^\times/K^{\times 2}$ such that $\mathrm{rk} \, J^t(K) = 1$.
\end{theorem}

If we order $(a_1, \dots, a_{2g + 1}) \in O_K^{2g + 1}$ by the maximum of their Weil heights, then it is not difficult to show that 100\% of hyperelliptic curves $C$ of the special shape
$$
y^2 = (x - a_1) \cdots (x - a_{2g + 1})
$$
satisfy the hypotheses of our theorem. In particular, Theorem \ref{tMain} implies Theorem \ref{tFirstResult}. 

For $t\in K^{\times}/K^{\times 2},$ denote by $C^t$ the quadratic twist of $C$ by $t$. The infinite order point on $J^t(K)$ that we construct in the proof of Theorem \ref{tMain} has the form $[P] - [\infty]$ for some $P \in C^t(K)$. Combined with work of Stoll \cite{MR2264661}*{Theorem 5.1} (applied with $\Gamma = \mu_2$, $G = J^t(K)$ and $T = C^t(K)$), we deduce the following result. 

\begin{corollary} 
\label{cor:stoll}
Take the notation of Theorem \ref{tMain} and suppose that $g\geq 2$. Then there are infinitely many elements $t\in K^{\times}/K^{\times 2}$ such that $C^t(K)$ contains exactly one pair of non-Weierstrass points. 
\end{corollary}

As a second corollary of Theorem \ref{tMain}, we obtain our next result.

\begin{corollary}
\label{cAb}
Let $K$ be a number field and let $d \geq 1$ be an integer. Then there are infinitely many absolutely simple abelian varieties $A$ of dimension $d$ with $\mathrm{rk} \, A(K) = 1$.
\end{corollary}

Taking $d = 1$, this recovers \cite{KP1}*{Corollary 1.1} and \cite{Zywina3}*{Theorem 1.1}: both of these papers independently prove that $1 \in \mathcal{R}(K, 1)$ for every number field $K$.

\subsection{Related results}
Theorem \ref{tMain} falls in the same theme as a string of recent results. Although the specific details vary greatly, key to all of these results is to combine additive combinatorics with descent. 

The first result of this type, which involves only elliptic curves and additionally imposes restrictions on the number field, can be found in Koymans--Pagano \cite{KPHilbert}. The main goal of the paper \cite{KPHilbert} is to settle Hilbert's tenth problem for finitely generated rings in the negative. The same conclusion was obtained soon thereafter in an elegant short paper by Alp\"oge--Bhargava--Ho--Shnidman \cite{Bhargava}, which cleverly exploits twists of certain hyperelliptic curves to obtain a rank growth result on abelian varieties, and Zywina \cite{Zywina3}, who proves a very general and strong rank growth result building on his works \cites{Zywina, Zywina2}. Zywina \cite{Zywina2} also proves the impressive result that $2 \in \mathcal{R}(\Q, 1)$, which was extended to $2 \in \mathcal{R}(\Q(i), 1)$ by Savoie \cite{Savoie}.

Zywina's most recent rank growth result \cite{Zywina3} implies in particular that for every number field $K$ there exist infinitely many non-isomorphic elliptic curves\footnote{We remark that Mazur--Rubin \cite{MR}, predating the recent flurry of results, had already shown that $0 \in \mathcal{R}(K, 1)$ using $2$-descent.} of rank $1$ over $K$, so $1 \in \mathcal{R}(K, 1)$. Independently of Zywina and at the same time, Koymans--Pagano \cite{KP1} also proved that $1 \in \mathcal{R}(K, 1)$. In fact, Koymans--Pagano show the much stronger result that a ``generic'' elliptic curve $E$ with full rational $2$-torsion has a quadratic twist of rank $1$. %It is the latter paper that is an important source of inspiration for our current work. 

Our current work is most closely related to the latter paper \cite{KP1} and follows the same high-level strategy and structure. Nevertheless, we emphasize that there are significant challenges in extending \cite{KP1} to hyperelliptic curves of arbitrary genus. Indeed, that work heavily relies on expressing Selmer ranks as a Markov chain process, which is possible thanks to work of Klagsbrun--Mazur--Rubin \cite{KMR}. However, as explained in the introduction to \cite{KMR}, this exploits several ``small number phenomena'' that fail for abelian varieties of dimension at least $2$. To salvage the strategy, we avoid working with Markov chains in \S \ref{ssMarkov} and instead work directly with Selmer groups of twists, which we are able to obtain sufficient control of by building on ideas of Harpaz \cite{Har19}. When specialised to genus $1$, this gives an alternative approach to \cite{KP1}*{\S3.2} which is somewhat shorter and simpler. 

%We have roughly broken our argument in two parts. In the first part, we %work with usual Selmer groups, and we 
%find a twist of $J$ with a carefully prescribed Selmer group. This argument makes use of certain places of multiplicative reduction (namely, those granted by genericity) and is inspired by earlier work of Harpaz \cite{Har19}. 
%\peter{Is there a good way here to summarize our innovation on the auxiliary twist? Our argument is both shorter, and so is our entire paper, and also more general! I can also try to say more about the challenges in the ``additive combinatorics'' step. But the main challenge in the additive combinatorics seems to be complicated juggling with indices so I am not sure how well I can sell it (except for explaining that the Markov chain breaks). Let me know if you want me to think harder about this and I would be happy to. 

%As discussed earlier today, I think we have good chances of acceptance but I am most scared about some silly report stating that this is a ``direct generalization'' or something of the sort: I think the quality of the result otherwise should definitely be strong enough. For example, one thing that certainly should be mentioned in some way, I think, is that our argument when specialized to elliptic curves should be somewhat simpler than the argument in \cite{KP1} (even in the combinatorics part, but certainly in the auxiliary twist)} 
%In the second part, we will add in the primes from additive combinatorics.

Finally, we mention the recent work of Gajovic--Park \cite{GP}, which makes significant progress towards understanding several of the sets $\mathcal{R}(K, g)$. Of their many excellent results, we highlight in particular that $\{0, 1, 2\} \subseteq \mathcal{R}(K, 2)$ for every number field $K$. 

\subsection{Overview of the paper and method of proof} 
The main body of the paper consists of two sections. Section \ref{sPreliminaries} contains mostly background material and is divided in four subsections. In \S \ref{ssDescent} we recall the basic theory of $2$-descent on Jacobians of hyperelliptic curves with full rational $2$-torsion. In \S \ref{ssMult} we state some facts about places of multiplicative reduction. These places will play an important technical role in our arguments, as our Selmer group computations will heavily exploit some of their special properties. We will also state some technical lemmas in this subsection: the first lemma recalls a result of Yelton \cite{MR4316889} that lets us deduce that $J$ is absolutely simple under some mild hypotheses, while the second lemma will be used to show that the leading constant in our additive combinatorics result does not vanish. 

In \S \ref{ssQuadratic} we introduce some general theory regarding quadratic twists. At the end of this subsection, we prove a critical lemma that gives sufficiently strong control on the torsion subgroup of our Jacobian. This will ensure that the rational point that we will construct in Section \ref{sMain} is not torsion. Finally, in \S \ref{ssMarkov} we detail a rather explicit mechanism for changing Selmer ranks, based on work of Harpaz \cite{Har19} and Mazur--Rubin \cite{MR}.

The main argument of our paper can be found in Section \ref{sMain}. Our proof consists of two main ingredients. The first ingredient is executed in \S \ref{sHarpaz}. In this subsection, we construct a twist of $J$ that has a prescribed Selmer group (here we prescribe both the dimension, but also the local behavior of our Selmer elements at certain places of multiplicative reduction). This is a rather delicate argument, partly inspired by ideas of Harpaz \cite{Har19}. 

The second ingredient is a combination of additive combinatorics with descent and can be found in \S \ref{sKai}. In common with the works \cites{KPHilbert, KP1} of Koymans--Pagano, the idea is to quadratic twist $C$ to get a new curve $C'$ with equation
$$
m(n - a_1m) \cdots (n - a_{2g + 1}m) y^2 = (x - a_1) \cdots (x - a_{2g + 1}),
$$
where $n, m \in O_K$ are free parameters to be chosen later. By construction, the choice of twist ensures that $C'$ has the rational point $P = (n/m, 1/m^{g + 1})$. For all but finitely many twists, the point $[P] - [\infty] \in \mathrm{Jac}(C')(K)$ is non-torsion, hence the rank of $\mathrm{Jac}(C')$ is at least $1$. Additive combinatorics, specifically a theorem of Kai \cite{Kai}, is then used to carefully control the prime factors of 
$$
m(n - a_1m) \cdots (n - a_{2g + 1}m).
$$
This careful control is then used to simultaneously perform descent, from which we deduce that the rank is at most $1$ (and hence must be $1$). For this step, we need a rather carefully prescribed Selmer group before we add in the prime factors from additive combinatorics, and it is here that we rely on our result in \S \ref{sHarpaz} in an essential way.

\subsection*{Acknowledgements}
The first author gratefully acknowledges the support of the Dutch Research Council (NWO) through the Veni grant ``New methods in arithmetic statistics''.
\section{Preliminaries}
\label{sPreliminaries}

\subsection{Descent on Jacobians}
\label{ssDescent}
Let $C$ be the hyperelliptic curve given by the equation
\begin{equation} 
\label{eq:hyp_curve}
C: y^2 = (x - a_1) \cdots (x - a_{2g + 1}),
\end{equation}
for distinct $a_1, \dots, a_{2g + 1} \in K$. Denote by $\infty$ the point at infinity on $C$. If we write $J = \mathrm{Jac}(C)$, we have an exact sequence
\begin{equation} 
\label{Kummer_sequence}
0 \rightarrow J[2] \rightarrow J \xrightarrow{\cdot 2} J \rightarrow 0.
\end{equation}
Let $D_i = [(a_i, 0)] - [\infty]$. Then the $D_i$ span $J[2]$ with the single relation 
$$
\sum_{i = 1}^{2g + 1} D_i = 0.
$$
Denoting by $e: J[2]\times J[2] \to \mu_2$ the Weil pairing on $J[2],$ we have (cf.~\cite{MR2964027}*{Section 5.2.2})
\[
e(D_i, D_j) = 
\begin{cases}1~~&~~i=j,\\ 
-1~~&~~\textup{otherwise.} 
\end{cases}
\]
We shall frequently use this formula without further mention.

Write $\Sigma \subseteq \mu_2^{2g + 1}$ for the subspace of $\mu_2^{2g + 1}$ such that the product of all the coordinates equals $1$. Then we have an identification $J[2] \cong \Sigma$ given by
$$
D \mapsto \left(e(D, D_1), \dots, e(D, D_{2g + 1})\right).
$$
This isomorphism also has an explicit inverse given by sending a tuple $(\alpha_1, \dots, \alpha_{2g + 1}) \in \Sigma$ to $\sum_{i = 1}^{2g + 1} f(\alpha_i) D_i$, where $f: \mu_2 \cong \mathbb{F}_2$ is the unique isomorphism between these groups.

The inclusion $\Sigma \hookrightarrow \mu_2^{2g+1}$, combined with Kummer theory, induces an isomorphism from $H^1(G_K, \Sigma)$ to the subgroup of $(K^{\times}/K^{\times2})^{2g+1}$ consisting of tuples for which the product of all coordinates is a square. Combined with the isomorphism $J[2] \cong \Sigma$ above, this allows us to view elements of $H^1(G_K, J[2])$ as $2g + 1$ tuples of elements of $K^{\times}/K^{\times 2}$. Since the product of all coordinates of such a tuple is a square, any $2g$ coordinates uniquely determine the last. 

% https://fse.studenttheses.ub.rug.nl/21636/1/mMATH_2020EvinkT.pdf

\begin{lemma} 
\label{descent_lemma}
Suppose that $D = \sum_P n_P [P] \in J$ and suppose that $D$ has no Weierstrass point in its support. Then the connecting map $\delta\colon J(K) \rightarrow H^1(G_K, J[2]) \cong H^1(G_K, \Sigma)$ is given by
$$
D \mapsto \left(\prod_P (x(P) - a_1)^{n_P}, \dots, \prod_P (x(P) - a_{2g + 1})^{n_P}\right).
$$
On the %Weierstrass 
points $D_i\in J[2]$, we have
$$
D_i \mapsto \left(a_i - a_1, \dots, \prod_{\substack{1 \leq j \leq 2g + 1 \\ j \neq i}} (a_i - a_j), \dots, a_i - a_{2g + 1}\right),
$$
where the product $\prod_{\substack{1 \leq j \leq 2g + 1 \\ j \neq i}} (a_i - a_j)$ appears in the $i$th coordinate.
\end{lemma}

\begin{proof}
See \cite{MR1829626}*{Section 4}. 
\end{proof}

\subsection{Multiplicative places and geometric endomorphism rings}
\label{ssMult}
For $i \neq j$, write $D_{ij} = D_i + D_j$. For this subsection, we suppose that $a_i\in O_K$ for all $i$, as will also be the case in Section \ref{sMain}. 

\begin{mydef} 
\label{dMultiplicative}
Let $w$ be a finite place of $K$ not dividing $2$ and let $i,j$ be distinct elements of $\{1, \dots, 2g + 1\}$. We say that $w$ is a \textit{multiplicative place of type} $\{i,j\}$ if, for all $k\neq l$, we have 
\[\textup{val}_{w}(a_k-a_l)=\begin{cases}1~~&~~\{k,l\}=\{i,j\},\\ 0~~&~~\textup{otherwise.}\end{cases}\]
%\begin{itemize}
%\item each $a_k$ is integral at $\mathfrak{p}$,
%\item for $k\neq l$, we have 
%\[\textup{val}_{\mathfrak{p}}(a_k-a_l)=\begin{cases}1~~&~~\{k,l\}=\{i,j\},\\ 0~~&~~\textup{otherwise.}\end{cases}\]
%\end{itemize}
\end{mydef}

\begin{remark} 
\label{rem:component_group}
If $w$ is a multiplicative place of type $\{i,j\}$, then $J$ has semistable reduction at $w$ with toric rank $1$. That is, the identity component of the special fibre of the N\'{e}ron model of $J$ over $O_{K_w}$ is an extension of an abelian variety by a dimension one torus. The component group $\Phi_w$ of the geometric special fibre of the N\'{e}ron model is isomorphic to $\mathbb{Z}/2\mathbb{Z}$. The resulting homomorphism $J(K_w)\to \mathbb{Z}/2\mathbb{Z} \cong \{\pm 1\}$, given by reducing points to $\Phi_w$, sends $P$ to $e(P,D_{ij}).$ See `Proof of Theorem 1.3 assuming Theorem 2.8' in \cite{Har19}, or \cite{MR4316889}*{Section 6} for details. 
\end{remark}

\begin{lemma} 
\label{lem:selmer_condtions_at_mult_places}
Let $w$ be a multiplicative place of type $\{i,j\}$. Let $\alpha = (\alpha_1, \dots, \alpha_{2g + 1}) \in (K_w^{\times}/K_w^{\times 2})^{2g+1}$ with product equal to $1$, which we view as an element of $H^1(G_{K_w}, J[2])$. % as in Lemma \ref{descent_lemma}. %Let $\mathscr{S}_{\mathfrak{p}}$ denote the image of the connecting map $\delta:J(K_\mathfrak{p})\to H^1(G_{K_{\mathfrak{p}}}, J[2])$. 
\begin{itemize}
\item[(i)] Suppose that $\textup{val}_{w}(\alpha_k) \equiv 0 \bmod 2$ for all $k$. Then $\alpha\in \delta(J(K_{w}))$ if and only if $\alpha_i\alpha_j\in K_{w}^{\times 2}$. 
\item[(ii)] The subgroup $\delta(J(K_{w}))$ of $H^1(G_{K_w}, J[2])$ is generated by $H^1_{\textup{ur}}(G_{K_{w}}, J[2])\cap \delta(J(K_{w}))$ and $\delta(D_i)$. 
\item[(iii)] The image of the valuation homomorphism $\delta(J(K_{w}))\to (\mathbb{Z}/2\mathbb{Z})^{2g + 1}$ given by 
$$
\alpha \mapsto \big(\textup{val}_{w}(\alpha_k) \bmod 2 \big)_{k = 1}^{2g + 1}
$$
is one dimensional, generated by the tuple having $1$ in the $i$th and $j$th coordinates, and zeroes elsewhere.
\end{itemize}
\end{lemma}

\begin{proof}
(i): The long exact sequence for cohomology associated to the sequence \eqref{Kummer_sequence} shows that $\alpha\in \delta(J(K_{w}))$ if and only if $\alpha$ has trivial image in $H^1(G_{K_w}, J)$. Since $\alpha$ is unramified, its image in $H^1(G_{K_w}, J)$ lies in $H^1_{\textup{ur}}(G_{K_w}, J)$. By \cite{MR2261462}*{Proposition I.3.8}, reduction to the special fibre of the N\'{e}ron model induces an isomorphism 
\[
H^1(\textup{Gal}(K_{w}^\textup{ur}/K_w), J(K_w^{\textup{ur}}))\cong H^1(G_{k_w},\Phi_{w}),
\]
where $k_w$ denotes the residue field at $w$. 

Since the Galois action on $\Phi_w\cong \mathbb{Z}/2\mathbb{Z}$ is trivial, evaluation at the Frobenius element $\textup{Frob}_w$ gives an isomorphism $H^1(G_{k_w},\Phi_{w})\cong \Phi_{w}.$ Thus $\alpha$ lies in $\delta(J(K_{w}))$ if and only if $\alpha(\textup{Frob}_w)\in J[2]$ has trivial image in $\Phi_w.$ By Remark \ref{rem:component_group}, this is the case if and only if $e(\alpha(\textup{Frob}_w), D_{ij}) = 1$. Using the discussion surrounding Lemma \ref{descent_lemma} to translate this to a condition on the coordinates $\alpha_k$ gives the part.

(ii): By (i), the group $H^1_{\textup{ur}}(G_{K_{w}}, J[2]) \cap \delta(J(K_{w}))$ corresponds to the group of $2g + 1$ tuples $(\alpha_1, \dots, \alpha_{2g+1})$ for which $\textup{val}_{w}(\alpha_k)$ is even for all $k$, and for which (the product of all coordinates is a square and) $\alpha_i\alpha_j$ is a square. This space is $2g - 1$ dimensional. Since $w$ is odd, we have (see e.g. \cite{MR1370197}*{Proposition 3.9})  
\[
\dim \delta(J(K_{w})) = \dim J(K_w)/2J(K_w)= \dim J[2] = 2g.
\] 
From Lemma \ref{descent_lemma} it follows that the $w$-adic valuation of the $i$th and $j$th coordinates of $\delta(D_i)$ is odd. In particular, $\delta(D_i) \notin H^1_{\textup{ur}}(G_{K_{w}}, J[2])$. Counting dimensions gives (ii). 

(iii): follows from (ii) and the description of $\delta(D_i)$ afforded by Lemma \ref{descent_lemma}. 
\end{proof}

\begin{remark} \label{rem:weil_pairing_interpretation}
We can rephrase Lemma \ref{lem:selmer_condtions_at_mult_places}(i) as follows: if $\alpha$ is unramified at $w$, then we have $\alpha \in \delta(J(K_w))$ if and only if the quadratic character $e(\alpha,D_{ij})$ has trivial restriction to $G_{K_{w}}$.
\end{remark}

The following result of Yelton ensures that the Jacobians we work with are absolutely simple.

\begin{lemma} 
\label{lem:2_structure_end}
Suppose there exists a set $\{w_1, \dots, w_{2g}\}$ of places of $K$ such that, for each $i \in \{1, \dots, 2g\}$,   $w_i$ is multiplicative of type $\{i, 2g + 1\}$. Then $\textup{End}(J_{\overline{K}}) = \mathbb{Z}$. In particular, $J$ is absolutely simple. 
\end{lemma}

\begin{proof}
This is \cite{MR4316889}*{Theorem 4.1}, combined with the proof of \cite{MR4316889}*{Corollary 1.4}. %Under the additional assumption that $a_j \in O_K$ for all $j$, this is \cite{MR4316889}*{Theorem 4.1}, combined with the proof of \cite{MR4316889}*{Corollary 1.4}. One readily checks that all that is used in the proofs of the cited results is that $a_j$ is $\mathfrak{p}_i$-integral for all $i$ and all $j$. This latter condition is guaranteed by the assumption that each $\mathfrak{p}_i$ is a multiplicative prime of type $\{i, 2g + 1\}$. 
\end{proof}

When we apply our additive combinatorics machinery later, it will be important that the resulting linear forms can take enough values modulo $w$ for certain multiplicative places $w$  for $J$. We shall accomplish this through our next result. 

\begin{lemma}
\label{lFiniteField}
Let $n \in \Z_{\geq 1}$. Then there exists $B > 0$ such that the following holds for all prime powers $q > B$. 

Let $\alpha_1, \dots, \alpha_n \in \mathbb{F}_q^\times$ and $\beta_1, \dots, \beta_n \in \mathbb{F}_q$. Define for each $1 \leq i \leq n$ the affine linear form
$$
L_i(X) = \alpha_i X + \beta_i.
$$
We assume that $\alpha_i \beta_j - \alpha_j \beta_i \neq 0$ for all distinct $i$ and $j$. Let $\epsilon_1, \dots, \epsilon_n \in \mathbb{F}_q^\times$. Then there exist $\mu \in \mathbb{F}_q$ and $z_1, \dots, z_n \in \mathbb{F}_q^\times$ such that $L_i(\mu) = \epsilon_i \cdot z_i^2$ for all $i$.
\end{lemma}

\begin{proof}
Denote by $\chi: \mathbb{F}_q^\times \rightarrow \{\pm 1\}$ the unique surjective homomorphism. Define $Z$ to be the set of $\mu \in \mathbb{F}_q$ such that $L_i(\mu) = 0$ for some $i$. Then $|Z| = n$, and we need to show that
$$
\sum_{\mu \in \mathbb{F}_q - Z} \prod_{i = 1}^n \left(1 + \chi(\epsilon_i L_i(\mu))\right) > 0.
$$
But we have
$$
\sum_{\mu \in \mathbb{F}_q - Z} \prod_{i = 1}^n \left(1 + \chi(\epsilon_i L_i(\mu))\right) = q - n + \sum_{\varnothing \subset S \subseteq \{1, \dots, n\}} \chi\left(\prod_{i \in S} \epsilon_i\right) \sum_{\mu \in \mathbb{F}_q - Z} \chi\left(\prod_{i \in S} L_i(\mu)\right).
$$
By our assumptions $\alpha_i \neq 0$ and $\alpha_i \beta_j - \alpha_j \beta_i \neq 0$, the equation $y^2 = \prod_{i \in S} L_i(x)$ defines a geometrically irreducible curve for $S\neq \varnothing$. By the Weil conjectures applied to this curve, each inner sum is bounded by $\ll_n q^{1/2}$. Therefore the lemma follows upon choosing a sufficiently large $B > 0$.
\end{proof}

\subsection{Quadratic twists}
\label{ssQuadratic}
Let $A$ be a principally polarised abelian variety over $K$. For a quadratic character $\chi: G_K \to \mu_2$, we denote by $A^\chi$ the quadratic twist of $A$ by $\chi$. This is an abelian variety over $K$, equipped with a $\overline{K}$-isomorphism $\phi:A\to A^\chi$ such that, for all $\sigma \in G_K$, the composition $\phi^{-1}\circ \sigma \phi \sigma^{-1}$ is multiplication by $\chi(\sigma)$. In particular, $A$ is isomorphic to $A^\chi$ over the quadratic extension cut out by $\chi$, and $A$ is absolutely simple if and only if $A^\chi$ is. The principal polarisation on $A$ descends via $\phi$ to a principal polarisation on $A^\chi$, via which we view $A^\chi$ as a principally polarised abelian variety. Moreover, $\phi$ induces an isomorphism of $G_K$-modules $A[2] \cong A^\chi[2]$ identifying the corresponding Weil pairings (see e.g.~\cite{MR3951582}*{Section 4.5} for details of the above). We will always identify $A^\chi[2]$ with $A[2]$ via $\phi$. In this way, for a field extension $F/K$, we view the connecting map $A^\chi(F) \to H^1(G_F, A^\chi[2])$ as taking values in $H^1(G_F, A[2])$. Similarly, we view both of the $2$-Selmer groups $\textup{Sel}_2(A)$ and $\textup{Sel}_2(A^\chi)$ as subgroups of $H^1(G_K, A[2])$. 

If $d \in K^{\times}$ (or $K^{\times}/K^{\times 2}$), then we  write $\chi_d \in H^1(G_K, \mu_2)$ for the quadratic character given by Kummer theory. We will write both $A^d$ and $A^{\chi_d}$ interchangeably. 

Given $x\in H^0(G_K, A[2])$, we can map $x$ to $H^1(G_K, A[2])$ via the connecting map $\delta: A(K) \to H^1(G_K, A[2])$, the result of which we denote $\delta(x)$. Alternatively, having identified $A[2]$ with $A^\chi[2]$, we can map $x$ to $H^1(G_K, A[2])$ via the connecting map $A^\chi(K) \to H^1(G_K, A^\chi[2])$, the result of which, to avoid ambiguity, we denote $\delta_\chi(x)$ (or $\delta_d(x)$ if $\chi = \chi_d$). 

The relationship between $\delta(x)$ and $\delta_\chi(x)$ is described by the following basic result.

\begin{lemma} 
\label{changes_under_twist_lemma} 
For all $x \in H^0(G_K, A[2])$, we have $\delta_\chi(x) = \delta(x) + \chi \cup x$.
\end{lemma}

\begin{proof}
Let $y \in A[4]$ be such that $2y=x$. For all $\sigma \in G_K$, we have
\[
\delta_\chi(x)(\sigma) = \phi^{-1}(\sigma \phi(y)-\phi(y))=\chi(\sigma)\sigma y - y =
\begin{cases}
\delta(x)(\sigma)~~&~~\chi(\sigma) = 1,
\\ \delta(x)(\sigma)+x~~&~~\chi(\sigma) = -1,
\end{cases}
\]
from which the result follows.
\end{proof}

%Using the explicit formulae in Lemma \ref{descent_lemma}, one directly deduces our next result. If $\chi$ is a quadratic character, then we write $J^\chi$ for the quadratic twist of $J$ by $\chi$, and we denote by $\delta_\chi$ the resulting connecting map.

\begin{example} 
\label{ex:hyp_curves} 
Let $C$ be the hyperelliptic curve given by \eqref{eq:hyp_curve}. For $d\in K^{\times}$, we denote by $C^d$ the quadratic twist of $C$ corresponding to the character $\chi=\chi_d$, which is given by the equation\footnote{Alternatively, by the equivalent equation $dy^2 = (x - a_1) \cdots (x - a_{2g+1})$.}
\[
C^d: y^2 = (x - da_1) \cdots (x - da_{2g+1}).
\]
The Jacobian of $C^d$ is $J^d$. Indeed, the $\overline{K}$-isomorphism $f: C \to C^d$ given by $(x, y) \mapsto (dx, \sqrt{d}d^{g}y)$ induces a $\overline{K}$-isomorphism $f_*: J \to \textup{Jac}(C^d)$, which satisfies the required identity $f_*^{-1}\circ \sigma f_* \sigma^{-1} = \chi_d(\sigma)$ for all $\sigma \in G_K$ (as follows from the fact that the hyperelliptic involution $(x,y)\mapsto (x,-y)$ on $C$ induces multiplication by $-1$ on $J$). The resulting isomorphism $J[2]\cong J^d[2]$ identifies $D_i$ with $[(da_i,0)]-[\infty]$. The fact that this identifies the Weil pairings can now be seen from the explicit formula given in Section \ref{ssDescent}. Note that the isomorphism $J[2]\cong J^d[2]$ is compatible with the isomorphisms $J[2]\cong \Sigma$ and $J^d[2]\cong \Sigma$ given in Section \ref{ssDescent}. Consequently, the explicit description of the connecting map given in Lemma \ref{descent_lemma} is compatible with quadratic twisting. In particular, one sees from that lemma that $\delta_d(D_i)$ corresponds to the tuple
\[
\Bigg(d(a_i - a_1), \dots, \prod_{\substack{1 \leq j \leq 2g + 1 \\ j \neq i}} (a_i - a_j), \dots, d(a_i - a_{2g + 1})\Bigg),
\]
which recovers Lemma \ref{changes_under_twist_lemma} in this case. 
\end{example}

\begin{remark} 
\label{rem:unramf_twist_of_mult}
In the notation of Example \ref{ex:hyp_curves}, suppose that $a_i\in O_K$ for all $i$. Let $w$ be a multiplicative place of type $\{i, j\}$ for $J$ and let $\chi$ be a quadratic character unramified at $w$. By strong approximation, we can find a non-zero element $d \in O_K$ with $\chi = \chi_d$ and $\textup{val}_w(d) = 0$. For this choice of $d$, it follows from the above discussion that $w$ is also a multiplicative place of type $\{i, j\}$ for $J^d$. (One could give a more intrinsic definition of multiplicative place that makes this true without carefully choosing the element $d$ such that $\chi=\chi_d$, but we opt not to do this to avoid complicating Definition \ref{dMultiplicative}.) 
\end{remark}

The next two lemmas give some basic, and largely well-known, results about how the local images $\delta(A(K_v))$ vary under quadratic twist.

\begin{lemma} 
\label{lem:restriction_inertia}
Let $v$ be a finite place of $K$ not dividing $2$. Suppose that $A$ has good reduction at $v$ and $\chi$ is ramified at $v$.  Then the map $\delta_\chi:A(K_v^{\textup{ur}})[2]\to H^1(G_{K^{\textup{ur}}_v},A[2])$ is injective. Moreover, we have $\delta(A^\chi(K_v))=\delta_\chi(A(K_v)[2]).$
\end{lemma}

\begin{proof}
Since $A$ has good reduction at $v$, the  $G_{K_v}$-module $A[4]$ is unramified. In particular, the connecting map $\delta:A(K_v^{\textup{ur}})[2]\to H^1(G_{K^{\textup{ur}}_v},A[2])$ is trivial. The injectivity of $\delta_\chi:A(K_v^{\textup{ur}})[2]\to H^1(G_{K^{\textup{ur}}_v},A[2])$ now follows from Lemma \ref{changes_under_twist_lemma} and the fact that $\chi$ is ramified at $v$.   

Since $v$ does not divide $2$, we have 
$$
\dim \delta^\chi(A(K_v)) = \dim A^\chi(K_v)/2A^\chi(K_v) = \dim A(K_v)[2].
$$
By the first part, the map $\delta_\chi:A(K_v)[2]\to H^1(G_{K_v},A[2])$ is injective, since it is so after restriction to $G_{K^{\textup{ur}}_v}$. Thus $\delta_\chi(A(K_v)[2])$ is a subgroup of $\delta(A^\chi(K_v))$ of full dimension. 
\end{proof}

\begin{lemma} 
\label{lema:Selmer_conditions}
Let $v$ be a finite place of $K$ not dividing $2$. Suppose that $A$ has good reduction at $v$ and let $\chi: G_{K_v} \rightarrow \mathbb{F}_2$ be a quadratic character.  

\begin{itemize}
\item[(i)] We have $\delta(A(K_v)) = H^1_{\textup{ur}}(G_{K_v}, A[2])$. 
\item[(ii)] If $\chi$ is unramified at $v$, then $\delta(A^\chi(K_v)) = \delta(A(K_v))$.
\item[(iii)] If $\chi$ is ramified at $v$, then $\delta(A^\chi(K_v)) \cap \delta(A(K_v)) = 0$.
\end{itemize}
\end{lemma}

\begin{proof}
The first part is standard. See e.g.~\cite{PR}*{Remark 4.12 and Proposition 4.13}. The second part follows from (i) and the fact that,  when $\chi$ is unramified, the quadratic twist $A^\chi$ has good reduction at $v$ also. For the final part, combining (i) with Lemma \ref{lem:restriction_inertia} gives 
\[
\delta(A^\chi(K_v))\cap \delta(A(K_v))=\delta_\chi(A(K_v)[2])\cap H^1_{\textup{ur}}(G_{K_v},A[2]).
\]
Lemma \ref{lem:restriction_inertia} shows that no non-zero element of $\delta_\chi(A(K_v)[2])$ has trivial restriction to $G_{K_v^{\textup{ur}}}$, so the latter intersection is trivial. (See \cite{MR3519097}*{Lemma 4.3} for an alternative proof of this part.)
\end{proof}

In order to guarantee that the rational point that we will construct is not torsion, we will use our next result.

\begin{lemma}[Torsion lemma]
\label{lTorsion}
For all but finitely many twists $d\in K^\times/K^{\times 2}$, we have $A^d(K)_{\textup{tors}} = A^d[2]$.
\end{lemma}

\begin{proof}
Let $L$ be the maximal multiquadratic extension of $K$. We claim that $A(L)_{\textup{tors}}$ is finite. To see this, pick a nonarchimedean place $v$ of $K$ such that $A$ has good reduction at $v$. Let $\overline{v}$ be a place of $L$ extending $v$. By Krasner's lemma, the completion $L_{\overline{v}}$ of $L$ at $\overline{v}$ is a finite extension of $K_v$. In particular, it has finite residue field $\mathbb{F}$, say of characteristic $p$. Since $A$ has good reduction at $v$, the natural reduction map on points gives an injection $A(L_{\overline{v}})_{\textup{tors}}'\hookrightarrow \overline{A}(\mathbb{F})$, where $\overline{A}$ denotes the special fibre of the N\'{e}ron model of $A$ over $O_{K_v}$, and the prime in the superscript denotes prime-to-$p$ torsion. It follows that the prime-to-$p$ torsion in $A(L)$ is finite. Choosing a second place of good reduction and residue characteristic different to $p$ completes the proof of the claim. 

Now let $d_1, \dots, d_n$ be distinct elements of $K^\times/K^{\times 2}$. For each prime $l \neq 2$ the natural map 
\[
\bigoplus_{i=1}^n A^{d_i}(K)[l^\infty] \to A(L_n) \subseteq A(L)
\]
is injective, where $L_n$ denotes the subextension of $L/K$ given by adjoining the square roots of each $d_i$. This follows, for example, from the existence of an isogeny 
\[
\prod_{\chi}A^\chi \to \textup{Res}_{L_n/K}A
\]
of degree a power of $2$, where the product runs over all characters of $\textup{Gal}(L_n/K)$, and $\textup{Res}_{L_n/K}$ denotes restriction of scalars from $L_n$ to $K$, so that $(\textup{Res}_{L_n/K} A)(K) = A(L_n)$. See \cite{MR2462119}*{Proposition 5.3} for details. Combining this with the claim above, we see that only finitely many twists of $A$ have non-trivial prime-to-$2$ torsion. 

%then this is a consequence of the natural direct sum decomposition of $l$-adic Tate modules 
%\textup{Ind}_{K}^{L_n}T_l(A)\cong \bigoplus_{\chi}T_l(A) \otimes \chi,
%\]
%where the sum on the right runs over all characters of $\textup{Gal}(L_n/K)$. Here the isomorphism is of $G_K$-modules.
Finally, let $d\in K^{\times}/K^{\times 2}$, and suppose that $\chi_d$ is ramified at a prime $\mathfrak{p}\nmid 2$ of good reduction for $A$ (as is the case for all but finitely many $d$). It follows from Lemma \ref{lem:restriction_inertia} that $A^d(K_{\mathfrak{p}}^{\textup{ur}})[2^\infty] = A^d(K_{\mathfrak{p}}^{\textup{ur}})[2].$  Indeed, injectivity of the map $\delta_d: A(K_{\mathfrak{p}}^{\textup{ur}})[2]\to H^1(G_{K_{\mathfrak{p}}^{\textup{ur}}}, A[2])$ implies that $A^d$ has no elements of order $4$ defined over $K_{\mathfrak{p}}^{\textup{ur}}$. This gives the result.
%By N\'{e}ron--Ogg--Shafarevich, the inertia subgroup $I_{\mathfrak{p}}$ of $G_{K_{\mathfrak{p}}}$ acts trivially on $A[2^\infty]$. Consequently, any $\sigma \in I_{\mathfrak{p}}$ acts on $A^d[2^\infty]$ as multiplication by $\chi_d(\sigma)$. Since $\chi_d$ is ramified at $\mathfrak{p}$, its restriction to $I_{\mathfrak{p}}$ is non-trivial, hence $A^d(K_{\mathfrak{p}}^{\textup{ur}})[2^\infty]=A^d[2],$ where $K_{\mathfrak{p}}^{\textup{ur}}$ denotes the maximal unramified extension of $K_\mathfrak{p}$. This gives the result.
\end{proof}

\subsection{Variation of Selmer groups}
\label{ssMarkov}
Let $A$ be a principally polarised abelian variety over $K$ and let $g = \dim A$. Let $\chi$ and $\chi'$ be quadratic characters and let $X$ be a finite set of places. Write 
\[
V_{X}^\chi = \textup{im}\left(\textup{Sel}_2(A^\chi) \stackrel{\oplus_{v\in X}\res_v}{\longrightarrow} \bigoplus_{v\in X} \frac{\delta(A^\chi(K_v))}{\delta(A^\chi(K_v))\cap \delta(A^{\chi'}(K_v))}\right),
\]
and define $V_{X}^{\chi'}$ similarly by interchanging $\chi$ and $\chi'$ in the displayed equation. 

The following result is due to Harpaz, building on work of Mazur--Rubin \cite{MR}*{Section 3}.

\begin{lemma} 
\label{lem:mazur_rubin_harpaz}
Suppose that $\delta(A^\chi(K_v))=\delta(A^{\chi'}(K_v))$ for all $v\notin X$. Then 
\[
\dim \textup{Sel}_2(A^{\chi'}) = \dim \textup{Sel}_2(A^{\chi}) + \dim V_X^{\chi'} - \dim V_X^\chi.
\]
Moreover, we have 
\[
\dim V_X^\chi + \dim V_X^{\chi'} \leq \sum_{v \in X} \dim \frac{\delta(A^\chi(K_v))}{\delta(A^\chi(K_v)) \cap \delta(A^{\chi'}(K_v))}.
\]
\end{lemma}

\begin{proof}
This is \cite{Har19}*{Lemma 3.27}, combined with the displayed equation immediately preceding that result.
\end{proof}

\begin{remark}  
\label{rmk:typical_application}
In Lemma \ref{lem:mazur_rubin_harpaz}, suppose that $X$ consists entirely of finite places $v$ not dividing $2$. Suppose also that, for all places $v$ in $X$, $A$ has good reduction,  and $\chi$ is unramified while $\chi'$ is ramified. From Lemma \ref{lema:Selmer_conditions} we see that $V_X^\chi$ is the image of $\textup{Sel}_2(A^\chi)$ in $\bigoplus_{v\in X}H^1(G_{K_v},A[2])$. Moreover, since $\dim \delta(A^\chi(K_v)) = \dim A(K_v)[2]$ for all $v\in X$, we have  
\[
\dim V_X^\chi + \dim V_X^{\chi'} \leq 2g|X|.
\]
\end{remark}

While the situation described in Remark \ref{rmk:typical_application} will be sufficient for most purposes, in Section \ref{sHarpaz} below  we will  apply Lemma \ref{lem:mazur_rubin_harpaz} when $X$ contains a carefully chosen multiplicative place. The following lemma will facilitate this. See  \cite{Har19}*{Lemma 3.25} for a generalisation of this result to arbitrary abelian varieties. 

\begin{lemma} 
\label{lHarpazChange}
Let $C$ be the hyperelliptic curve given by equation \eqref{eq:hyp_curve}, let $J = \textup{Jac}(C)$ and suppose that $a_1, \dots a_{2g + 2} \in O_K$. Let $1 \leq i < j \leq 2g + 1$ and let $w$ be a multiplicative place of type $\{i, j\}$. Suppose that $\chi$ and $\chi'$ are quadratic characters unramified at $w$ such that precisely one of $\res_{w}(\chi)$ and $\res_{w}(\chi')$ is non-trivial. Then we have 
\begin{equation} 
\label{eq:intersectin_multiplicative}
\delta(J^\chi(K_w)) \cap \delta(J^{\chi'}(K_w)) = \delta(J^\chi(K_w)) \cap H^1_{\textup{ur}}(G_{K_{w}}, J[2]).
\end{equation}
Moreover, we have 
\begin{equation} 
\label{eq:dimension_multiplicative}
\dim \frac{\delta(J^\chi(K_w))}{\delta(J^\chi(K_w)) \cap \delta(J^{\chi'}(K_w))} = 1.
\end{equation}
\end{lemma}

\begin{proof}
As explained in Example \ref{ex:hyp_curves} and  Remark \ref{rem:unramf_twist_of_mult}, we can assume without loss of generality that $w$ is also a multiplicative place of type $\{i, j\}$ for both $C^\chi$ and $C^{\chi'}$. The claimed description of $\delta(J^\chi(K_w))\cap \delta(J^{\chi'}(K_w))$ then follows from Lemma \ref{lem:selmer_condtions_at_mult_places} and the explicit description of the connecting map given in Lemma \ref{descent_lemma}. Having established \eqref{eq:intersectin_multiplicative}, the equality \eqref{eq:dimension_multiplicative} follows from Lemma \ref{lem:selmer_condtions_at_mult_places}(iii).
\end{proof}

%\begin{remark} 
%For an explicit description of the intersection $\delta(J^\chi(K_w))\cap H^1_{\textup{ur}}(G_{K_{w}}, J[2])$, see Lemma \ref{lem:selmer_condtions_at_mult_places}.  
%\end{remark}
\section{Twisting and additive combinatorics}  
\label{sMain}
The goal of this section is to prove Theorem \ref{tMain}. As already indicated in the introduction, this implies both Theorem \ref{tFirstResult} and Corollaries \ref{cor:stoll} and \ref{cAb}. Moreover, it already follows from Lemma \ref{lem:2_structure_end} that $J$ is absolutely simple, so it suffices to establish the rank $1$ part of the theorem.

Throughout this section, we fix a number field $K$ and a hyperelliptic curve $C$ of the shape
$$
C: y^2 = (x - a_1) \cdots (x - a_{2g + 1}),
$$
where $a_1, \dots, a_{2g + 1} \in O_K$ are distinct. Define $B > 0$ to be the real number provided by Lemma \ref{lFiniteField} applied with $n := 2g + 2$. We will assume that the Jacobian $J := \mathrm{Jac}(C)$ is $B$-generic. 

Take $T$ to be any finite set of places including all $2$-adic and infinite places, all places of bad reduction of $J$, all primes $\mathfrak{p}$ with $N_{K/\Q}(\mathfrak{p}) \leq 2g + 1$ and all primes $\mathfrak{p}$ such that $\text{val}_\mathfrak{p}(a_i - a_j) \neq 0$.

\subsection{Finding prime elements}
\label{sKai}
For notational convenience, we shall drop the last coordinate in our Selmer groups for this subsection only, i.e.~we shall frequently and implicitly exploit the isomorphism $\Sigma \cong \mu_2^{2g}$ given by projection on the first $2g$ coordinates, thus identifying $H^1(G_K,J[2])$ with $(K^{\times}/K^{\times 2})^{2g}$.

\begin{mydef}
We say that an element $t \in K^\times/K^{\times 2}$ is a suitable twist if there exists $d \in K^\times/K^{\times 2}$ and prime elements $p, q_1, \dots, q_{2g + 2}$ such that $t = d pq_1 \cdots q_{2g + 2}$ and moreover:
\begin{enumerate}
\item[$(P1)$] we have that
\begin{itemize}
\item the elements $p,q_1, \dots, q_{2g + 2}$ are pairwise coprime, coprime to $T$ and coprime to all primes where $d$ has odd valuation,
\item $d$ is a non-square modulo $p$, $p$ splits completely in $K(J[4])/K,$ and $p$ is a square locally at all places in $T$,
\item $pq_1\cdots q_{2g + 2}$ is a square locally at all places $v\in T$ and all places where $d$ has odd valuation;
\end{itemize}
\item[$(P2)$] $\mathrm{Sel}_{2}(J^d)$ has dimension $4g^2 + 4g - 1$ and is spanned by $\delta_d(J[2])$ and linearly independent elements
$$
\mathbf{z}_i := ((z_{i, 1}, \dots, z_{i, 2g}))_{1 \leq i \leq 4g^2 + 2g - 1}
$$
with $\textup{res}_p(\textbf{z}_i)=1$ for all $i$. Write $T'$ for the union of $T$, $p$, and all primes where $d$ has odd valuation. Then for all $1 \leq i_1, i_2 \leq 2g$, all $1 \leq j \leq 2g$ and all $k \leq i_1 \leq 2g$, we have 
\begin{equation}
\label{eZiConf}
\prod_{v \in T'} (z_{(i_1 - 1) 2g + i_2, j}, q_k)_v = - 1 \Longleftrightarrow i_1 = k \text{ and } i_2 = j.
\end{equation}
Moreover, we have for all $1 \leq i \leq 2g - 1$, all $1 \leq j \leq 2g$ and all $1 \leq k \leq 2g$
\begin{equation}
\label{eZiConf2}
\prod_{v \in T'} (z_{4g^2 + i, j}, q_k)_v = 1.
\end{equation}
Finally, we have for all $1 \leq i \leq 2g - 1$ and all $1 \leq j \leq 2g$
\begin{equation}
\label{eZiConf3}
\prod_{v \in T'} (z_{4g^2 + i, j}, q_{2g + 1})_v = - 1 \Longleftrightarrow j \in \{i, i + 1\};
\end{equation}
\item[$(P3)$] we have $\mathrm{rk} \, J^t(K) > 0$.
\end{enumerate}
\end{mydef}

\begin{lemma} 
\label{lem:inj_res_map}
Let $t = dpq_1 \cdots q_{2g + 2}$ be a suitable twist. Then the natural restriction map
\[
\textup{Sel}_2(J^d) \to \bigoplus_{v\in \{p, q_1, \dots, q_{2g + 2}\}} H^1(G_{K_v}, J[2])
\]
is injective. 
\end{lemma}

\begin{proof}
By assumption, we have $\textup{res}_p(\mathbf{z}_i) = 1$ for all $1 \leq i \leq 4g^2 + 2g - 1$. Since $p$ splits completely in $K(J[4])/K$, Lemma \ref{changes_under_twist_lemma} gives $\textup{res}_p(\delta_d(x)) = \textup{res}_p(\chi_d) \cup x$ for all $x \in J[2]$. Since $d$ is a non-square modulo $p$, it follows that restriction to $G_{K_p}$ maps $\delta_d(J[2])$ injectively into $H^1(G_{K_p}, J[2])$. Consequently, it suffices to show that the image of the set $\{\mathbf{z}_i : 1 \leq i \leq 4g^2 + 2g - 1\}$ in the group 
\[
\bigoplus_{v \in \{q_1, \dots, q_{2g + 2}\}} H^1(G_{K_v}, J[2])
\] 
is linearly independent.

By equations \eqref{eZiConf} and \eqref{eZiConf2} and an application of Hilbert reciprocity, we see that for all $1 \leq j \leq 2g$, we have
$$
\left(\frac{\mathbf{z}_i}{q_{j}}\right) = \mathbf{1} \quad \text{ for } i \geq 2 j g + 1, \quad \quad \quad \quad \left(\frac{\mathbf{z}_{2(j-1)g + i}}{q_j}\right) = \mathbf{e}_i \quad \text{ for } 1 \leq i \leq 2g,
$$
where $\mathbf{1}$ is the unique vector in $\{\pm 1\}^{2g}$ consisting entirely of ones, and $\mathbf{e}_i$ is the unique vector in $\{\pm 1\}^{2g}$ which is $- 1$ precisely on the $i$th coordinate. Similarly, it follows from Hilbert reciprocity and equation \eqref{eZiConf3} that, for all $1 \leq i \leq 2g - 1$, we have
$$
\left(\frac{\mathbf{z}_{4g^2 + i}}{q_{2g + 1}}\right) = \mathbf{e}_i \cdot \mathbf{e}_{i + 1}.
$$
One readily deduces the sought linear independence from these equations. 
% e.g.~by successively projecting a potential linear dependence relation to $H^1(G_{K_{q_i}}, J[2])$, for $i = 1, 2, \dots, 2g+1$ in turn.
\end{proof}

\begin{theorem}
Assume that there exist infinitely many $t \in K^\times/K^{\times 2}$ that are suitable. Then Theorem \ref{tMain} holds.
\end{theorem}

\begin{proof}
We claim that
\begin{align}
\label{eFinalSel}
\dim \mathrm{Sel}_2(J^t/K) \leq 2g + 1
\end{align}
for all suitable $t$. Before we proceed with the proof of the claim, we explain how the claim implies the theorem. To this end, we recall that $J^t(K)[2] \cong \mathbb{F}_2^{2g}$, and therefore we have the inequality
$$
2g + \mathrm{rk} \, J^t(K) \leq \dim \mathrm{Sel}_2(J^t/K).
$$
It follows from $(P3)$ and equation \eqref{eFinalSel} that $\mathrm{rk} \, J^t(K) = 1$, as desired.

It remains to prove equation \eqref{eFinalSel}. Since $p$ and $q_1, \dots, q_{2g + 2}$ are primes of good reduction for $J$, and since $pq_1\cdots q_{2g + 2}$ is assumed to be a square locally at all places in $T$ and all places where $\chi_d$ ramifies, we can apply Lemma \ref{lem:mazur_rubin_harpaz} with $\chi=\chi_d$, $\chi'=\chi_t$ and $X = \{p, q_1, \dots, q_{2g + 2}\}$ to obtain 
\begin{equation} 
\label{eq:Selmer_change_app}
\dim \textup{Sel}_2(J^t) = 4g^2 + 4g - 1+\dim V_X^{\chi_t} - \dim V_X^{\chi_d},
\end{equation}
along with the inequality $\dim V_X^{\chi_d} + \dim V_X^{\chi_t} \leq 4g^2 + 6g$ (cf.~Lemma \ref{lema:Selmer_conditions}(ii) and Remark \ref{rmk:typical_application}). By Lemma \ref{lem:inj_res_map} we have 
\[
\dim V_X^{\chi_d} = \dim \mathrm{Sel}_2(J^d) = 4g^2 + 4g - 1,
\]
hence $\dim V_X^{\chi_t} \leq 2g+1$. Combining this with equation \eqref{eq:Selmer_change_app} gives the result.
\end{proof}

Henceforth, we write $W=\{w_1, \dots, w_{4g^2 + 2g - 1}\}$ and $W'=\{w_1', \dots, w_{2g}'\}$ for the sets of places guaranteed by the genericity of $J$ (see Definition \ref{def:places_in_aux_twist}). In this subsection, we will make extensive use of the places in $W$. The places in $W'$ will only become relevant in Section~\ref{sHarpaz}.

\begin{mydef}
We say that $\kappa \in K^\times/K^{\times 2}$ is an auxiliary twist if there exists $d \in K^\times/K^{\times 2}$ and a prime element $p$ such that $\kappa = dp$ and moreover:
\begin{enumerate}
\item[$(K1)$] we have that
\begin{itemize}
\item the element $d$ is unramified at $p, w_1, \dots, w_{4g^2 + 2g - 1}$ and is a non-square modulo $p$, 
\item $p$ splits completely in $K(J[4])/K,$ and $p$ is a square locally at all places in $T$,
\end{itemize}
\item[$(K2)$] $\mathrm{Sel}_{2}(J^d)$ has dimension $4g^2 + 4g - 1$ and is spanned by $\delta_d(J[2])$ and linearly independent elements
$$
\mathbf{z}_i := ((z_{i, 1}, \dots, z_{i, 2g}))_{1 \leq i \leq 4g^2 + 2g - 1}
$$
with $\textup{res}_p(\textbf{z}_i)=1$ for all $i$. Furthermore, we have:
\begin{itemize}
\item for all $1 \leq i_1, i_2 \leq 2g$ and all $1 \leq j_1, j_2 \leq 2g$ and all $1 \leq k \leq 2g$
\begin{equation}
\label{eZiSpec1}
\textup{val}_{w_{(i_1 - 1) 2g + i_2}}(z_{(j_1 - 1) 2g + j_2, k}) \equiv 1 \bmod 2 \Longleftrightarrow i_1 = j_1 \textup{ and } i_2 = j_2 = k,
\end{equation}
\item for all $1 \leq i_1, i_2 \leq 2g$ and all $1 \leq j \leq 2g - 1$ and all $1 \leq k \leq 2g$
\begin{equation}
\label{eZiSpec2}
\textup{val}_{w_{(i_1 - 1) 2g + i_2}}(z_{4g^2 + j, k}) \equiv 0 \bmod 2,
\end{equation}
\item for all $1 \leq i \leq 2g - 1$ and all $1 \leq j_1, j_2 \leq 2g$ and all $1 \leq k \leq 2g$ 
% \marginpar{This can be chosen to be anything essentially}
\begin{equation}
\label{eZiSpec3}
\textup{val}_{w_{4g^2 + i}}(z_{(j_1 - 1)2g + j_2,k}) \equiv 0 \bmod 2,
\end{equation}
\item for all $1 \leq i \leq 2g - 1$ and all $1 \leq j \leq 2g - 1$ and all $1 \leq k \leq 2g$
\begin{equation}
\label{eZiSpec4}
\textup{val}_{w_{4g^2 + i}}(z_{4g^2 + j, k}) \equiv 1 \bmod 2 \Longleftrightarrow i = j \textup{ and } k \in \{j, j + 1\}.
\end{equation}
\end{itemize}
\end{enumerate}
\end{mydef}

\noindent Our next result relies crucially on additive combinatorics, and is modelled after the proof of \cite{KP1}*{Theorem 3.5}. One of the more challenging new ingredients of this paper is to show the existence of an auxiliary twist. We shall dedicate our next subsection entirely to this task.

\begin{theorem}
Suppose that there exists an auxiliary twist $\kappa$. Then there exist infinitely many suitable $t$.
\end{theorem}

\begin{proof}
The proof proceeds in three steps. Our ultimate goal will be to apply Kai's result\footnote{A simplified (but weaker) version of his result can be found in \cite{KPHilbert}*{Theorem A.8}. It does not apply directly to our situation as we require $2g + 2$ linear forms instead of $4$ affine linear forms. However, the obvious generalization to $2g + 2$ affine linear forms of \cite{KPHilbert}*{Theorem A.8} is true and follows from Kai's theorem.} \cite{Kai}*{Theorem 12.1}. This result requires as input affine linear forms and a convex region.

In our first step, we construct $2g + 2$ affine linear forms, and check that these are admissible (which just means that there is no obstruction modulo a finite place to representing prime elements). In our second step, we construct the convex region and asymptotically calculate its volume. Therefore the main term in \cite{Kai}*{Theorem 12.1} dominates, and it follows that our linear forms represent primes with the expected asymptotic formula. In the third step, we use these prime elements to construct $t$ and verify that $t$ is suitable.

\subsubsection*{Affine linear forms}
Fix some $\kappa = dp$ satisfying $(K1)$ and $(K2)$. Then $\mathrm{Sel}_2(J^d)$ has a basis spanned by $\delta_d(J[2])$ and linearly independent elements
$$
\mathbf{z}_i := ((z_{i, 1}, \dots, z_{i, 2g}))_{1 \leq i \leq 4g^2 + 2g - 1}
$$
satisfying $\res_p(\mathbf{z}_i) = 1$ for all $i$ and satisfying equations \eqref{eZiConf}, \eqref{eZiConf2} and \eqref{eZiConf3}. 

By strong approximation, we may fix an integral representative for $\kappa$ that is coprime to the places $w_i$. Define $N$ to be the product of all finite places in $T - W$. %$T - \{w_1, \dots, w_{4g^2 + 2g - 1}\}$.
We also fix a generator $\rho$ for $8N^{h_K}$. By applying strong approximation again, there exists $\lambda \in O_K$ coprime to $\kappa$ such that
\begin{align}
&p \lambda \equiv \square \bmod \mathfrak{p} \quad \quad \text{ for all } \mathfrak{p} \not \in T \text{ with } \mathrm{val}_\mathfrak{p}(d) \equiv 1 \bmod 2 \label{eLambdaChoice} \\
&\lambda \equiv 1 \bmod \rho. \label{eLambdaChoice2}
\end{align}
We now pick $\mu_0 \in O_K$ such that
\begin{equation}
\label{eFinalMu}
\left(\frac{\rho^2 \kappa \mu_0 + \lambda}{w_j}\right) = 
\begin{cases}
1 &\text{if } 1 \leq j \leq 4g^2, \ 1 + \lfloor \frac{j - 1}{2g} \rfloor \equiv j \bmod 2g, \\
- 1 &\text{if } 1 \leq j \leq 4g^2, \ 1 + \lfloor \frac{j - 1}{2g} \rfloor \not \equiv j \bmod 2g, \\
- 1 &\text{if } 4g^2 < j \leq 4g^2 + 2g - 1.
\end{cases}
\end{equation}
Observe that the condition $1 + \lfloor \frac{j - 1}{2g} \rfloor \equiv j \bmod 2g$ for $1 \leq j \leq 4g^2$ is equivalent to demanding that $j$ is of the shape $(i_1 - 1) 2g + i_2$ with $1 \leq i_1 = i_2 \leq 2g$.

We define the linear functions 
\begin{alignat*}{3}
&M_{1}(X) &&:= &&\rho^2 \kappa X - a_1 \rho^4 \kappa^2 \mu_0 - a_1 \rho^2 \kappa \lambda + 1 \\
& &&\ \ \vdots && \\
&M_{2g + 1}(X) &&:= &&\rho^2 \kappa X - a_{2g + 1} \rho^4 \kappa^2 \mu_0 - a_{2g + 1} \rho^2 \kappa \lambda + 1.
\end{alignat*}
We will now aim towards applying Lemma \ref{lFiniteField} to the linear functions $M_i(X)$ and the multiplicative places $w_j$ with $1 \leq j \leq 4g^2 + 2g - 1$. However, before we do so, we warn the reader that at a multiplicative place $w$ of type $\{i, j\}$ the linear functions $M_{i}(X)$ and $M_{j}(X)$ coincide modulo $w$, which means that their reductions certainly take the same value up to squares. Therefore we proceed as follows.

At each place $w_{(i - 1) 2g + j}$, which is a multiplicative place of type $\{j, 2g + 1\}$ by definition, Lemma \ref{lFiniteField} allows us to prescribe the square class of each $M_k(\mu)$ modulo $w_{(i - 1) 2g + j}$ for each $1 \leq k \leq 2g$, and then the square class of $M_{2g + 1}(\mu)$ modulo $w_{(i - 1) 2g + j}$ is determined (and equals $M_j(\mu)$). Similarly, for the places $w_{4g^2 + j}$ with $1 \leq j \leq 2g - 1$, which is a multiplicative place of type $\{j, j + 1\}$, we prescribe the square class of $M_k(\mu)$ for $k \in \{1, \dots, j, j + 2, \dots, 2g + 1\}$, and then the square class of $M_j(\mu)$ equals that of $M_{j + 1}(\mu)$.

Upon doing so and by applying strong approximation, we find $\mu_1 \in O_K$ such that
\begin{itemize}
\item for all $1 \leq i \leq 2g$ and all $1 \leq j \leq 2g$ and all $1 \leq k \leq 2g$ we have
\begin{equation}
\label{eMu1}
\left(\frac{M_{k}(\mu_1)}{w_{(i - 1) 2g + j}}\right) = 
\begin{cases}
1 &\text{if } i \neq k, \\
- 1 &\text{if } i = k,
\end{cases}
\end{equation}
\item for all $1 \leq i \leq 2g$ and all $1 \leq j \leq 2g$ we have
\begin{equation}
\label{eMu2}
\left(\frac{M_{2g + 1}(\mu_1)}{w_{(i - 1) 2g + j}}\right) = 
\begin{cases}
1 &\text{if } i \neq j, \\
- 1 &\text{if } i = j.
\end{cases}
\end{equation}
Recall that the square class of $M_{2g + 1}(\mu_1)$ must agree with the square class of $M_j(\mu_1)$ for multiplicative places $w$ of type $\{j, 2g + 1\}$ (which is the type of $w_{(i - 1) 2g + j}$). Upon comparing equations \eqref{eMu1} and \eqref{eMu2}, one sees that this is indeed the case,
\item for all $1 \leq j \leq 2g - 1$ and all $1 \leq k \leq 2g$ we have
\begin{equation}
\label{eMu3}
\left(\frac{M_{k}(\mu_1)}{w_{4g^2 + j}}\right) = 1.
\end{equation}
Clearly, $M_j(\mu_1)$ and $M_{j + 1}(\mu_1)$ have the same square class, as required for a multiplicative place of type $\{j, j + 1\}$,
\item for all $1 \leq j \leq 2g - 1$
\begin{equation}
\label{eMu4}
\left(\frac{M_{2g + 1}(\mu_1)}{w_{4g^2 + j}}\right) = - 1.
\end{equation}
\end{itemize}
Let $m$ be a generator of the ideal $(w_1 \cdots w_{4g^2 + 2g - 1})^{h_K}$. We define $2g + 2$ affine linear forms $L_1, \dots, L_{2g + 2} \in O_K[X, Y]$ via
\begin{alignat*}{3}
&L_1(X, Y) &&:= &&\rho^2 \kappa (mX + \mu_1) - a_1 \rho^2 \kappa (\rho^2 \kappa (mY + \mu_0) + \lambda) + 1 \\
& &&\ \ \vdots && \\
&L_{2g + 1}(X, Y) &&:= &&\rho^2 \kappa (mX + \mu_1) - a_{2g + 1} \rho^2 \kappa (\rho^2 \kappa (mY + \mu_0) + \lambda) + 1 \\
&L_{2g + 2}(X, Y) &&:= &&\rho^2 \kappa (mY + \mu_0) + \lambda.
\end{alignat*}
It will be useful later to note that, for all $1 \leq i \leq 2g + 1$, we have
\begin{equation} 
\label{eq:dependence_on_i}
L_i(X, Y) = \rho^2\kappa(mX + \mu_1) + 1 - a_i \rho^2 \kappa L_{2g+2}(X,Y)
\end{equation}
and 
\begin{equation}
\label{eq:congruence_mod_w}
L_i(X,Y) \equiv M_i(\mu_1) \bmod w_j \quad \quad \textup{ for all } 1 \leq j \leq 4g^2 + 2g - 1.
\end{equation}
We will finish the first part of our construction by demonstrating that the affine linear forms $L_1, \dots, L_{2g + 2}$ are admissible.

\begin{lemma}
\label{lAdmissible}
For every prime ideal $\mathfrak{p}$, there exist $u, v \in O_K$ such that
$$
L_i(u, v) \not \equiv 0 \bmod \mathfrak{p} \quad \quad \textup{ for all } 1 \leq i \leq 2g + 2.
$$
\end{lemma}

\begin{proof}
The proof is similar to \cite{KP1}*{Lemma 3.6}. Firstly, suppose that $\mathfrak{p}$ divides $\rho \kappa$. In this case, we pick $u = v = 0$. Then we check that
$$
L_i(u, v) \equiv 1 \bmod \mathfrak{p}
$$
for all $1 \leq i \leq 2g + 1$. Moreover, we have
$$
L_{2g + 2}(u, v) \equiv \lambda \bmod \mathfrak{p}.
$$
Since $\lambda$ satisfies $\lambda \equiv 1 \bmod \rho$ and since $\lambda$ is coprime to $\kappa$ by construction, this is non-zero modulo $\mathfrak{p}$.

Secondly, suppose that $\mathfrak{p}$ divides $m$.  We pick $u = v = 0$ again, so from \eqref{eq:congruence_mod_w} we have
$$
L_i(u, v) \equiv  M_{i}(\mu_1) \bmod \mathfrak{p}
$$
%$$
%L_i(u, v) \equiv \rho^2 \kappa \mu_1 - a_i \rho^4 \kappa^2 \mu_1' - a_i \rho^2 \kappa \lambda + 1 \equiv M_{i, \mu_1'}(\mu_1) \bmod \mathfrak{p}
%$$
for all $1 \leq i \leq 2g + 1$, while straight from the definition of $L_{2g+2}(X,Y)$ we have
$$
L_{2g + 2}(u, v) \equiv \rho^2 \kappa \mu_0 + \lambda \bmod \mathfrak{p}.
$$
By construction of $\mu_0$ and $\mu_1$, we see that these elements are non-zero modulo $\mathfrak{p}$. 

Thirdly, suppose that $\mathfrak{p}$ is any place not dividing $\rho m \kappa$. In particular, it follows from our construction of $T$ that $N_{K/\Q}(\mathfrak{p}) > 2g + 1$. We start by picking any $v$ such that
$$
\rho^2 \kappa mv + \rho^2 \kappa \mu_0 + \lambda \not \equiv 0 \bmod \mathfrak{p}.
$$
Because $N_{K/\Q}(\mathfrak{p}) > 2g + 1$, it follows that there exists $u$ such that
$$
u \bmod \mathfrak{p} \not \in \left\{\frac{-\rho^2 \kappa \mu_1 - 1 + a_i \rho^2 \kappa \left(\rho^2 \kappa (mv + \mu_0) + \lambda\right)}{\rho^2 \kappa m} \bmod \mathfrak{p} : 1 \leq i \leq 2g + 1\right\}.
$$
Having chosen these particular $u$ and $v$, it is readily verified that 
$$
L_i(u, v) \not \equiv 0 \bmod \mathfrak{p} \quad \quad \textup{ for all } 1 \leq i \leq 2g + 2,
$$
ending the proof of the lemma.
\end{proof}

\subsubsection*{A convex region} 
Our next goal is to construct the convex region $\Omega \subseteq \R^{2n}$ to which we apply \cite{Kai}*{Theorem 12.1}. Roughly speaking, we take the largest region that makes all the prime values $q_i$ attained by our linear forms totally positive. This argument is similar to \cite{KP1}*{Construction of a convex region}, although we shall now give full details for completeness.

To this end, we fix an integral basis $\{\omega_1, \dots, \omega_n\}$ of $K$. For $1\leq i\leq 2g+2$, we define the affine linear form $\widetilde{L_i}: \Z^{2n} \rightarrow O_K$ by 
$$
(x_1, \dots, x_n, y_1, \dots, y_n) \mapsto L_i(x_1 \omega_1 + \dots + x_n \omega_n, y_1 \omega_1 + \dots + y_n \omega_n).
$$
For every real embedding $\sigma: K \xhookrightarrow{} \R$, this gives an induced affine linear form $\sigma \circ \widetilde{L_i}: \Z^{2n} \rightarrow \R$. By viewing $\Z^{2n}$ inside $\R^{2n}$ in the natural way, this also extends uniquely to an affine linear form $\sigma \circ \widetilde{L_i}: \R^{2n} \rightarrow \R$. This allows us to define
$$
\Omega := \left\{(\mathbf{x}, \mathbf{y}) \in \R^{2n} : \sigma(\widetilde{L_i}(\mathbf{x}, \mathbf{y})) > 0 \textup{ for all } \sigma: K \xhookrightarrow{} \R \textup{ and all } i \in \{1, \dots, 2g + 2\}\right\}.
$$

\begin{lemma}
\label{lRegion} 
There exists $C_1 > 0$ such that
$$
\Vol\left(\Omega \cap [-H, H]^{2n}\right) \sim C_1 H^{2n}
$$
as $H \rightarrow \infty$.
\end{lemma}

\begin{proof}
Take $\sigma: K \xhookrightarrow{} \mathbb{R}$ to be a real embedding. Our first goal will be to rewrite the $2g + 2$ inequalities
\begin{equation}
\label{eSystemIneq}
\sigma(\widetilde{L_i}(\mathbf{x}, \mathbf{y})) > 0 \text{ for all } i \in \{1, \dots, 2g + 2\}
\end{equation}
as two inequalities. To this end, denote by $i_\sigma$ the unique index $i$ where $\sigma(a_i \kappa)$ is maximal. From equation \eqref{eq:dependence_on_i}, one sees that the system of inequalities in \eqref{eSystemIneq} is equivalent to
$$
\sigma(\widetilde{L_{i_\sigma}}(\mathbf{x}, \mathbf{y})) > 0 \quad \quad \text{and} \quad \quad \sigma(\widetilde{L_{2g + 2}}(\mathbf{x}, \mathbf{y})) > 0.
$$
Write $\mathcal{M}_{i_\sigma, \sigma}$ and $\mathcal{M}_{2g + 2, \sigma}$ for the degree $1$ parts of $\widetilde{L_{i_\sigma}}$ and $\widetilde{L_{2g + 2}}$. We now apply \cite{KPHilbert}*{Lemma 5.6}. Since this result immediately implies our lemma, it remains to check that the hypotheses of \cite{KPHilbert}*{Lemma 5.6} are satisfied. Therefore we have to check that the system of linear forms $\{\mathcal{M}_{i_\sigma, \sigma} : \sigma \text{ real embedding}\} \cup \{\mathcal{M}_{2g + 2, \sigma} : \sigma \text{ real embedding}\}$ are linearly independent over $\R$. Define the column vector $\boldsymbol{\omega} = (\omega_1, \dots, \omega_n)^T$. Tracing through the definitions, we see that
\[
\mathcal{M}_{2g + 2, \sigma}(\mathbf{x}, \mathbf{y}) = \sigma(\rho^2\kappa m) \sigma(\boldsymbol{\omega})^T\mathbf{y}
\]
and
\[
\mathcal{M}_{i_\sigma, \sigma}(\mathbf{x},\mathbf{y}) = \sigma(\rho^2\kappa m) \sigma(\boldsymbol{\omega})^T\mathbf{x} -\sigma(a_{i_\sigma}\rho^2\kappa) \mathcal{M}_{2g + 2, \sigma}(\mathbf{x},\mathbf{y}).
\]  
In particular, the $\mathbb{R}$-span of the linear forms $\mathcal{M}_{i_\sigma, \sigma}$ and $\mathcal{M}_{2g + 2, \sigma}$ is the $\mathbb{R}$-span of the linear forms $\sigma(\boldsymbol{\omega})^T\mathbf{x}$ and $\sigma(\boldsymbol{\omega})^T\mathbf{y}$. The sought linear independence is now a consequence of the standard fact from algebraic number theory that the vectors $\{\sigma(\boldsymbol{\omega}) : \sigma \text{ real embedding}\}$ are linearly independent over $\mathbb{R}$.
\end{proof}

\subsubsection*{Constructing suitable twists} 
We now have all the ingredients to apply \cite{Kai}*{Theorem 12.1}. By Lemma \ref{lAdmissible} and Lemma \ref{lRegion}, we see that the main term in \cite{Kai}*{Theorem 12.1} dominates. Hence there are infinitely many tuples of prime elements 
$$
(q_{1, l}, \dots, q_{2g + 2, l})_{l \geq 1}
$$
attained by the linear forms $L_1, \dots, L_{2g + 2}$, where we further demand that the $q_{i, l}$ are coprime to $T'$ and that the ideals $(q_{i, l})$ are pairwise distinct.

We now fix one such tuple $(q_1, \dots, q_{2g + 2})$ and we define
$$
t := \kappa q_1 \cdots q_{2g + 2}.
$$
We will show that $t$ satisfies $(P1)$ and $(P2)$. Moreover, we will also show that $(P3)$ holds for all but finitely many tuples.

\paragraph{Verification of $(P1)$.} Note that $q_1, \dots, q_{2g + 2}$ are pairwise coprime and coprime to $T'$ by construction. This gives the first bullet point in $(P1)$. The second bullet point in $(P1)$ is a direct consequence of $(K1)$. It remains to verify the third and final bullet point in $(P1)$.

Let $\mathfrak{p} \not \in T$ be such that $\mathrm{val}_\mathfrak{p}(d) \equiv 1 \bmod 2$. Since every such $\mathfrak{p}$ divides $\kappa$, an inspection of the linear forms yields
\begin{align}
\label{eqiCong}
q_1 \equiv \cdots \equiv q_{2g + 1} \equiv 1 \bmod \mathfrak{p}, \quad \quad q_{2g + 2} \equiv \lambda \bmod \mathfrak{p}.
\end{align}
We also note that
\begin{equation}
\label{eqiCongN}
q_1 \equiv \cdots \equiv q_{2g + 1} \equiv 1 \bmod 8N, \quad \quad q_{2g + 2} \equiv \lambda \equiv 1 \bmod 8N
\end{equation}
by equation \eqref{eLambdaChoice2}. From this and equation \eqref{eLambdaChoice}, we conclude that the third bullet point holds at all places in $T - W$ and at all places $\mathfrak{p}$ satisfying $\mathrm{val}_\mathfrak{p}(d) \equiv 1 \bmod 2$. We now check the third bullet point at the places in $W$, which by $(K1)$ is equivalent to
\begin{align}
\label{eAtTClaim}
q_1 \cdots q_{2g + 2} \equiv \square \bmod w_j
\end{align}
for all $1 \leq j \leq 4g^2 + 2g - 1$. In order to prove \eqref{eAtTClaim}, we start by observing that, for each $1 \leq j \leq 4g^2 + 2g - 1$, a consequence of \eqref{eq:congruence_mod_w} and the definition of $L_{2g + 2}(X, Y)$ is 
\begin{equation}
\label{eqkMk}
q_k \equiv M_{k}(\mu_1) \bmod w_j \quad \text{ for } 1 \leq k \leq 2g + 1, \quad \quad q_{2g + 2} \equiv \rho^2 \kappa \mu_0 + \lambda \bmod w_j.
\end{equation}
First suppose that $j > 4g^2$. It then follows from \eqref{eMu3}, \eqref{eMu4} and \eqref{eFinalMu} that
$$
q_k \equiv \square \bmod w_j \quad \text{ for } 1 \leq k \leq 2g,
$$
while $q_{2g + 1}$ and $q_{2g + 2}$ are non-squares modulo $w_j$. This gives \eqref{eAtTClaim} for $j > 4g^2$. Now suppose that $1 \leq j \leq 4g^2$, and write $j = (i_1 - 1)2g + i_2$. Then it follows from \eqref{eMu1} and \eqref{eMu2} that $q_1 \cdots q_{2g + 1} \bmod w_j$ is a square if and only if $i_1 = i_2$. Hence we conclude that
$$
q_1 \cdots q_{2g + 2} \equiv \square \bmod w_j
$$
by equation \eqref{eFinalMu} and the observation immediately afterwards. We have now proven \eqref{eAtTClaim}.

\paragraph{Verification of $(P2)$.} We need to check equations \eqref{eZiConf}, \eqref{eZiConf2} and \eqref{eZiConf3}. Before we shall delve into this task, we will first claim that for all $1 \leq i \leq 4g^2 + 2g - 1$, for all $1 \leq j \leq 2g$ and all $1 \leq k \leq 2g + 1$
\begin{equation}
\label{eHilbertW}
\prod_{v \in T'} (z_{i, j}, q_k)_v = \prod_{l = 1}^{4g^2 + 2g - 1} (z_{i, j}, q_k)_{w_l}.
\end{equation}
We start by proving the claim, and we will then check equations \eqref{eZiConf}, \eqref{eZiConf2} and \eqref{eZiConf3}.

\paragraph{Proof of claim \eqref{eHilbertW}.} Recall that we have proven the following congruences in respectively equations \eqref{eqiCong} and \eqref{eqiCongN} 
$$
q_1 \equiv \cdots \equiv q_{2g + 1} \equiv 1 \bmod \mathfrak{p}, \quad \quad q_1 \equiv \cdots \equiv q_{2g + 1} \equiv 1 \bmod 8N
$$
for all prime ideals $\mathfrak{p} \not \in T$ satisfying $\mathrm{val}_{\mathfrak{p}}(d) \equiv 1 \bmod 2$. Moreover, we know that the $q_i$ are totally positive by construction of the region $\Omega$. Since $T' = T \cup \{p\} \cup \{\mathfrak{p} : \mathrm{val}_\mathfrak{p}(d) \equiv 1 \bmod 2\}$ and since $\res_p(\mathbf{z}_i) = 1$ by $(K1)$, it follows that
$$
\prod_{v \in T'} (z_{i, j}, q_k)_v = \prod_{v \in T} (z_{i, j}, q_k)_{w_l} = \prod_{l = 1}^{4g^2 + 2g - 1} (z_{i, j}, q_k)_{w_l},
$$
as was claimed.

\paragraph{Proof of equation \eqref{eZiConf}.} Recall that the Selmer elements $z_{i, j}$ have the local behavior at $w_l$ as specified by property $(K2)$, more specifically equations \eqref{eZiSpec1}, \eqref{eZiSpec2}, \eqref{eZiSpec3} and \eqref{eZiSpec4}. Moreover, we will repeatedly use equation \eqref{eHilbertW} to simplify the product of Hilbert symbols in equations \eqref{eZiConf}, \eqref{eZiConf2} and \eqref{eZiConf3}.

Take $1 \leq i_1, i_2 \leq 2g$, $1 \leq j \leq 2g$ and $k \leq i_1 \leq 2g$. We need to show that
$$
\prod_{l = 1}^{4g^2 + 2g - 1} (z_{(i_1 - 1) 2g + i_2, j}, q_k)_{w_l} = - 1 \Longleftrightarrow i_1 = k \text{ and } i_2 = j.
$$
By equations \eqref{eMu3} and \eqref{eqkMk} and the assumption $k \leq 2g$, this is equivalent to
$$
\prod_{l_1 = 1}^{2g} \prod_{l_2 = 1}^{2g} (z_{(i_1 - 1) 2g + i_2, j}, q_k)_{w_{(l_1 - 1)2g + l_2}} = - 1 \Longleftrightarrow i_1 = k \text{ and } i_2 = j.
$$
Note that $q_k$ is a unit locally at $w_{(l_1 - 1)2g + l_2}$. Hence we see that $(z_{(i_1 - 1) 2g + i_2, j}, q_k)_{w_{(l_1 - 1)2g + l_2}} = - 1$ is equivalent to
$$
\text{val}_{w_{(l_1 - 1)2g + l_2}}(z_{(i_1 - 1) 2g + i_2, j}) \equiv 1 \bmod 2 \ \ \ \text{ and } \ \ \ \left(\frac{q_k}{w_{(l_1 - 1)2g + l_2}}\right) = - 1.
$$
By equations \eqref{eZiSpec1}, \eqref{eMu1} and \eqref{eqkMk}, we see that the last two conditions are equivalent to $l_1 = i_1$ and $l_2 = i_2 = j$ and $k = l_1$. We conclude that
$$
\prod_{l_1 = 1}^{2g} \prod_{l_2 = 1}^{2g} (z_{(i_1 - 1) 2g + i_2, j}, q_k)_{w_{(l_1 - 1)2g + l_2}} = (z_{(i_1 - 1) 2g + i_2, j}, q_k)_{w_{(i_1 - 1)2g + i_2}} = (- 1)^{\mathbf{1}_{i_1 = k \text{ and } i_2 = j}},
$$
as desired.

\paragraph{Proof of equation \eqref{eZiConf2}.} Take $1 \leq i \leq 2g - 1$, $1 \leq j \leq 2g$ and $1 \leq k \leq 2g$. We need to show that
$$
\prod_{l = 1}^{4g^2 + 2g - 1} (z_{4g^2 + i, j}, q_k)_{w_l} = 1.
$$
But we have
$$
\prod_{l = 1}^{4g^2 + 2g - 1} (z_{4g^2 + i, j}, q_k)_{w_l} = \prod_{l = 4g^2 + 1}^{4g^2 + 2g - 1} (z_{4g^2 + i, j}, q_k)_{w_l} = 1,
$$
where the first equality follows from \eqref{eZiSpec2} and the second equality follows directly from equations \eqref{eMu3} and \eqref{eqkMk}.

\paragraph{Proof of equation \eqref{eZiConf3}.} Take $1 \leq i \leq 2g - 1$ and $1 \leq j \leq 2g$. We need to show that
$$
\prod_{l = 1}^{4g^2 + 2g - 1} (z_{4g^2 + i, j}, q_{2g + 1})_{w_l} = - 1 \Longleftrightarrow j \in \{i, i + 1\}.
$$
By equation \eqref{eZiSpec2}, we have
$$
\prod_{l = 1}^{4g^2 + 2g - 1} (z_{4g^2 + i, j}, q_{2g + 1})_{w_l} = \prod_{l = 4g^2 + 1}^{4g^2 + 2g - 1} (z_{4g^2 + i, j}, q_{2g + 1})_{w_l}.
$$
Note that $q_{2g + 1}$ is a unit locally at each $w_l$. Hence, for all $4g^2 + 1 \leq l \leq 4g^2 + 2g - 1$, we see that
$$
(z_{4g^2 + i, j}, q_{2g + 1})_{w_l} = - 1 \quad \Longleftrightarrow \quad \text{val}_{w_l}(z_{4g^2 + i, j}) \equiv 1 \bmod 2 \quad \text{ and } \quad \left(\frac{q_{2g + 1}}{w_l}\right) = - 1.
$$
%for all $4g^2 + 1 \leq l \leq 4g^2 + 2g - 1$. 
By equations \eqref{eZiSpec4}, \eqref{eMu4} and \eqref{eqkMk}, these last two conditions are equivalent to $j \in \{i, i + 1\}$ and $l = 4g^2 + i$. We conclude that
$$
\prod_{l = 4g^2 + 1}^{4g^2 + 2g - 1} (z_{4g^2 + i, j}, q_{2g + 1})_{w_l} = (z_{4g^2 + i, j}, q_{2g + 1})_{w_{4g^2 + i}} = (- 1)^{\mathbf{1}_{j \in \{i, i + 1\}}},
$$
as desired.

\paragraph{Verification of $(P3)$.} We now pick $(x_0, y_0) \in O_K^2$ such that $q_i = L_i(x_0, y_0)$. We set
$$
c := \rho^2 \kappa (mx_0 + \mu_1) + 1, \quad \quad \gamma := \rho^2 \kappa (\rho^2 \kappa (my_0 + \mu_0) + \lambda).
$$
We have  $L_i(x_0, y_0) = c - a_i \gamma$ for $i \in \{1, \dots, 2g + 1\}$ and $\gamma = \rho^2 \kappa L_{2g + 2}(x_0, y_0)$. Hence we conclude that
$$
t = \kappa q_1 \cdots q_{2g + 2} = \rho^{-2} \gamma (c - a_1\gamma) \cdots (c - a_{2g + 1}\gamma),
$$
so we see that the quadratic twist $C^t:ty^2=(x-a_1)\cdots (x-a_{2g+1})$ has the rational point $P := (c/\gamma, \rho/\gamma^{g + 1})$, which gives a rational point $[P] - [\infty]$ on $J^t$ as well. Since this rational point is clearly not $2$-torsion, it follows from Lemma \ref{lTorsion} that $(P3)$  holds for all but finitely many tuples.
\end{proof}

%ADAM: if $[P]-[\infty]$ were $2$-torsion then, since the hyperelliptic involution $\iota$ on $C^t$ induces multiplication by $-1$ on $J^t,$ we'd have $\iota(P)\sim P.$ Since $C^t$ has positive genus, this equation implies $\iota(P)=P,$ a contradiction.
%Also, one computes $\delta([P]-[\infty])=(\prod_{j\neq i}q_j)_{i=1}^{2g+1}$.

\subsection{Finding auxiliary twists}
\label{sHarpaz}

For this subsection, we return to writing elements of $H^1(G_K,J[2])$ as $2g+1$ tuples of elements of $K^{\times}/K^{\times 2}$ for which the product of all coordinates is a square, as in Section \ref{ssDescent}. Sometimes, however, we will implicitly switch between this representation and the standard representation of elements of $H^1(G_K,J[2])$ as $1$-cocycles valued in $J[2].$ For consistency,  in both cases we opt to use multiplicative notation for the group operation on $H^1(G_K,J[2])$. 

Recall that we write $W = \{w_1, \dots, w_{4g^2 + 2g - 1}\}$ and $W' = \{ w_1', \dots, w_{2g}'\}$ for the places guaranteed by the genericity of $J$ (see Definition \ref{def:places_in_aux_twist}).

\begin{theorem} 
\label{thm:almost_aux_twist}
There exists a quadratic character %are infinitely many quadratic characters 
$\chi$ with the following property:
\begin{itemize}
\item $\chi$ is unramified over $W \cup W'$, and is trivial at all places in $T - (W \cup W')$,
\item $\textup{Sel}_2(J^\chi)$ is generated by $\delta_\chi(J[2])$ and elements $\textbf{z}_1, \dots, \textbf{z}_{4g^2 + 2g - 1}$ satisfying \eqref{eZiSpec1}--\eqref{eZiSpec4},
\item $\chi$ is ramified at a place $v_0\notin T$ such that $\res_{v_0}(\textbf{z}_i) = 1$ for all $i \in \{1, \dots, 4g^2 + 2g-1\}$.
\end{itemize}
\end{theorem}

Before we proceed with the proof of Theorem \ref{thm:almost_aux_twist}, we will show that it readily implies the existence of an auxiliary twist.

\begin{corollary}
There exists an auxiliary twist $\kappa$. 
\end{corollary}

\begin{proof}
Take $\chi$ as in Theorem \ref{thm:almost_aux_twist}, and write $\chi=\chi_d$. Let $\mathfrak{m}$ be the formal product of $8$ and all places in $T$. Let $K_\mathfrak{m}/K$ be the corresponding ray class field. Thanks to the existence of the place $v_0$ in the statement of the theorem, by the Chebotarev density theorem there exists a prime ideal $\mathfrak{p}$ that splits completely in $K_\mathfrak{m}/K$, $K(J[4])/K$ and $K(\sqrt{z_{i,k}})/K$ for each coordinate $z_{i, k}$ of every element $\textbf{z}_i$, and is inert in $K(\sqrt{d})/K$. By construction, $\mathfrak{p}$ is principal, generated by some prime element $p$ that is a square locally at all places in $T$. 

The element $\kappa := dp$ is then an auxiliary twist. Indeed, all that remains to check is that $\dim \textup{Sel}(J^d) = 4g^2 + 4g - 1$. To see this, note that equations \eqref{eZiSpec1}--\eqref{eZiSpec4} imply that the elements $\mathbf{z}_1, \dots, \mathbf{z}_{4g^2 + 2g - 1}$ are linearly independent. That $\delta_d(J[2])$ has dimension $2g$ and intersects the span of $\{\mathbf{z}_1,\dots,\mathbf{z}_{4g^2 + 2g - 1}\}$ trivially follows from Lemma \ref{lem:restriction_inertia} applied to the place $v_0$.
\end{proof}

For a nonarchimedean place $w$ of $K$, and $\textbf{z}\in H^1(G_K, J[2])$ corresponding via Section \ref{ssDescent} to a $2g + 1$ tuple $(\alpha_1, \dots, \alpha_{2g+1})$ of elements of $K^{\times}/K^{\times 2},$ we define $\textup{val}_w(\textbf{z})\in (\mathbb{Z}/2\mathbb{Z})^{2g+1}$ to be the tuple $\big(\textup{val}_w(\alpha_i) \bmod 2\big)_{i=1}^{2g+1}.$

For $i, j \in \{1, \dots, 2g + 1\}$ distinct, denote by $E_{ij}$ the element of $(\mathbb{Z}/2\mathbb{Z})^{2g + 1}$ whose $i$th and $j$th coordinate is equal to $1$, and whose $k$th coordinate is equal to $0$ for all $k\notin \{i, j\}.$

\begin{proposition} 
\label{finding_selmer_basis}
Let $i$ and $j$ be distinct elements of $\{1, \dots, 2g + 1\}$ and let $w \in T$ be a multiplicative place of type $\{i, j\}$, with corresponding prime ideal $\mathfrak{p}_w$. Let $T'\supseteq T$ be a finite set of places. Then there exist distinct prime ideals $\mathfrak{p}_1$ and $\mathfrak{p}_2$ not in $T'$, and $a_{\mathfrak{p}_1\mathfrak{p}_2}, a_{\mathfrak{p}_1\mathfrak{p}_w} \in K^{\times}$, with the following properties:
\begin{itemize}
\item[(1)] the character $\varphi = \chi_{a_{\mathfrak{p}_1\mathfrak{p}_2}}$ is trivial at all $v\in T'- \{w\},$ is ramified at $\mathfrak{p}_1$ and $\mathfrak{p}_2$, and is unramified otherwise,
\item[(2)] the character $\psi=\chi_{a_{\mathfrak{p}_1\mathfrak{p}_w}}$ is trivial at all $v\in (T' \cup \{\mathfrak{p}_2\}) - \{w\}$, is ramified at $\mathfrak{p}_1$ and $\mathfrak{p}_w$, and is unramified otherwise,
\item[(3)] denote by $\textbf{z}_w$ the element $\psi\cup D_{ij}$ of $H^1(G_K, J[2])$.\footnote{This corresponds to the element 
\[
(1, \dots, 1, a_{\mathfrak{p}_1\mathfrak{p}_w}, 1, \dots, 1, a_{\mathfrak{p}_1\mathfrak{p}_w}, 1, \dots, 1)
\]
of $(K^{\times}/K^{\times 2})^{2g + 1}$, where the $a_{\mathfrak{p}_1\mathfrak{p}_w}$ terms appear in the $i$th and $j$th coordinate.}
Then $\textbf{z}_w$ is in the $2$-Selmer group of $J^{a_{\mathfrak{p}_1\mathfrak{p}_2}}$, is unramified at all $v\notin \{\mathfrak{p}_1, \mathfrak{p}_w\},$ and has trivial restriction to $G_{K_v}$ for all $v \in (T'\cup \{\mathfrak{p}_2\}) -\{w\}$. Moreover, we have $\textup{val}_w(\textbf{z}_w) = E_{i, j}$.
\end{itemize}
\end{proposition}

\begin{proof}
Let $\mathfrak{m}$ be the formal product of $8$ and all places $v\in T'- \{w\}$, and denote by $K_\mathfrak{m}/K$ the corresponding ray class field of conductor $\mathfrak{m}$. Choose a prime $\mathfrak{p}_1\notin T'$ such that $\textup{Frob}_{\mathfrak{p}_1}=\textup{Frob}_{\mathfrak{p}_w}^{- 1}$ in $\textup{Gal}(K_\mathfrak{m}/K)$ (the existence of infinitely many such primes is guaranteed by the Chebotarev density theorem). By construction, the ideal $\mathfrak{p}_1\mathfrak{p}_w$ is principal, with a generator $a_{\mathfrak{p}_1\mathfrak{p}_w}$ that is a square locally at all $v\in T'-\{w\}$. 

Having fixed $\mathfrak{p}_1$ and $a_{\mathfrak{p}_1\mathfrak{p}_w},$ choose a prime $\mathfrak{p}_2\notin T'\cup \{\mathfrak{p}_1\}$ such that the following hold:
\begin{itemize}
\item[(i)] we have $\textup{Frob}_{\mathfrak{p}_2} = \textup{Frob}_{\mathfrak{p}_w} =\textup{Frob}_{\mathfrak{p}_1}^{-1}$ in $\textup{Gal}(K_\mathfrak{m}/K)$, 
\item[(ii)] $\mathfrak{p}_2$ splits in $K(\sqrt{a_{\mathfrak{p}_1\mathfrak{p}_w}})/K$,
\item[(iii)] identifying $\textup{Gal}(K(\sqrt{a_i-a_j})/K)$ with $\{\pm 1\}$ in the obvious way, we have the equality 
\[\textup{Frob}_{\mathfrak{p}_2}=\textup{Frob}_{\mathfrak{p}_1}\cdot\left(\frac{\prod_{r\neq i, j}(a_i-a_r)}{\mathfrak{p}_w}\right)\]
in $\textup{Gal}(K(\sqrt{a_i-a_j})/K).$
\end{itemize}
Note that $K_\mathfrak{m}/K$ is unramified at $w$ and $\mathfrak{p}_1$, $K(\sqrt{a_i-a_j})/K$ is ramified at $w$ but unramified at $\mathfrak{p}_1$, and $K(\sqrt{a_{\mathfrak{p}_1\mathfrak{p}_w}})/K$ is ramified at $\mathfrak{p}_1.$ It follows from this that the Galois group of the compositum of these three extensions is the product of their individual Galois groups. Thus, there are infinitely many choices for the prime $\mathfrak{p}_2$ by the Chebotarev density theorem. Henceforth, we fix one such choice.

By (i) the ideal $\mathfrak{p}_1\mathfrak{p}_2$ is principal, with a generator $a_{\mathfrak{p}_1\mathfrak{p}_2}$ that is a square locally at all $v\in T'-\{w\}.$ By construction, properties (1) and (2) of the statement are satisfied by the associated characters $\varphi=\chi_{a_{\mathfrak{p}_1\mathfrak{p}_2}}$ and $\psi=\chi_{a_{\mathfrak{p}_1\mathfrak{p}_w}}$. We now turn to property (3) for the associated cocycle $\textbf{z}_{w}=\psi\cup D_{ij}$. That this is unramified at all $v\notin \{w, \mathfrak{p}_1\}$ and has trivial restriction to $G_{K_v}$ for all $v\in (T'\cup \{\mathfrak{p}_2\})-\{w\}$ follows from the corresponding properties of the character $\psi$. Further, since $\psi$ is ramified at $w$, we see that $\textup{val}_w(\textbf{z}_w)=E_{i, j}$. It remains to show that $\textbf{z}_w$ lies in the $2$-Selmer group of $J^{a_{\mathfrak{p}_1\mathfrak{p}_2}}$. That it satisfies the Selmer conditions at all $v\in (T'\cup \{\mathfrak{p}_2\})- \{w\}$ is immediate, and for the remaining places outside $\{w, \mathfrak{p}_1\}$, the cocycle $\textbf{z}_w$ is unramified at $v$ and $J^{a_{\mathfrak{p}_1\mathfrak{p}_2}}$ has good reduction at $v$. It remains to show that the Selmer conditions are satisfied at $w$ and $\mathfrak{p}_1$. We begin with the following claim.

\textbf{Claim:} we have 
\begin{equation} 
\label{restriction_at_w_char}
\res_w(\varphi) = \prod_{r \neq i, j} \res_w(\chi_{a_i-a_r}),
\end{equation}
and 
\begin{equation} 
\label{restriction_at_p_1_char}
\res_{\mathfrak{p}_1}(\varphi) = \res_{\mathfrak{p}_1}(\psi) \cdot \res_{\mathfrak{p}_1}(\chi_{- 1}) \cdot \prod_{r\neq i, j} \res_{\mathfrak{p}_1}(\chi_{a_i-a_r}). 
\end{equation}
We begin by establishing \eqref{restriction_at_w_char}. Note that the extension $K(\sqrt{a_i-a_j})/K$ is unramified outside $T'$, and is ramified at $w$ by our assumption that $\textup{val}_w(a_i-a_j)=1$. The product formula for Hilbert symbols gives
\[
1=(a_{\mathfrak{p}_1\mathfrak{p}_2},a_i-a_j)_w(a_{\mathfrak{p}_1\mathfrak{p}_2},a_i-a_j)_{\mathfrak{p}_1}(a_{\mathfrak{p}_1\mathfrak{p}_2},a_i-a_j)_{\mathfrak{p}_2} =\left(\frac{a_{\mathfrak{p}_1\mathfrak{p}_2}}{\mathfrak{p}_{w}}\right)\cdot \left(\frac{\prod_{r\neq i, j}(a_i-a_r)}{\mathfrak{p}_w}\right),
\]
where the second equality follows from property (iii). Equation \eqref{restriction_at_w_char} follows.

Next, again using the product formula for Hilbert symbols, we compute
\[
1=(a_{\mathfrak{p}_1\mathfrak{p}_w},a_{\mathfrak{p}_1\mathfrak{p}_2})_w(a_{\mathfrak{p}_1\mathfrak{p}_w},a_{\mathfrak{p}_1\mathfrak{p}_2})_{\mathfrak{p}_1}(a_{\mathfrak{p}_1\mathfrak{p}_w},a_{\mathfrak{p}_1\mathfrak{p}_2})_{\mathfrak{p}_2}=\left(\frac{a_{\mathfrak{p}_1\mathfrak{p}_2}}{\mathfrak{p}_{w}}\right)(a_{\mathfrak{p}_1\mathfrak{p}_w},a_{\mathfrak{p}_1\mathfrak{p}_2})_{\mathfrak{p}_1}.
\]
Rearranging this and multiplying by the trivial symbol $(a_{\mathfrak{p}_1\mathfrak{p}_w},-a_{\mathfrak{p}_1\mathfrak{p}_w})_{\mathfrak{p}_1}$ gives 
\[
\left(\frac{a_{\mathfrak{p}_1\mathfrak{p}_2}a_{\mathfrak{p}_1\mathfrak{p}_w}/\pi_1^2}{\mathfrak{p}_{1}}\right) = \left(\frac{a_{\mathfrak{p}_1\mathfrak{p}_2}}{\mathfrak{p}_{w}}\right) \left(\frac{- 1}{\mathfrak{p}_1}\right) = \left(\frac{\prod_{r\neq i, j}(a_i-a_r)}{\mathfrak{p}_w}\right) \left(\frac{- 1}{\mathfrak{p}_1}\right),
\]
where $\pi_1$ is a uniformiser at $\mathfrak{p}_1$ and the second equality was established above. The square root of $\prod_{r\neq i, j}(a_i-a_r)$ lies in the maximal quadratic subextension of $K_\mathfrak{m}$ and, by construction, the Frobenius element at $\mathfrak{p}_w$ in this extension agrees with that of $\mathfrak{p}_1,$ so we can replace $\mathfrak{p}_w$ with $\mathfrak{p}_1$ on the right hand side of the above displayed formula. This establishes \eqref{restriction_at_p_1_char}. 

We now show that $\textbf{z}_w$ satisfies the Selmer conditions at $\mathfrak{p}_1$. More precisely, we claim that 
\[
\res_{\mathfrak{p}_1}(\textbf{z}_w) = \res_{\mathfrak{p}_1}\delta_{a_{\mathfrak{p}_1\mathfrak{p}_2}}(D_{ij}).
\]
Indeed, we have
\begin{align*}
\res_{\mathfrak{p}_1}\delta_{a_{\mathfrak{p}_1\mathfrak{p}_2}}(D_{ij}) 
&=\res_{\mathfrak{p}_1}\delta(D_{ij}) \cdot \left(\res_{\mathfrak{p}_1}(\varphi) \cup D_{ij}\right) \\ 
&=\res_{\mathfrak{p}_1}\delta(D_{ij}) \cdot \res_{\mathfrak{p}_1}(\textbf{z}_w) \cdot \left((\res_{\mathfrak{p}_1}\chi_{-\prod_{r \neq i, j} (a_i-a_r)}) \cup D_{ij}\right),
\end{align*}
where the first equality follows from Lemma \ref{changes_under_twist_lemma} (upon converting between additive and multiplicative notation) and the second from \eqref{restriction_at_p_1_char}. Now Lemma \ref{descent_lemma} identifies $\delta(D_{ij}) = \delta(D_i) \cdot \delta(D_j)$ with the $2g+1$ tuple $\bf{x}$ with coordinates 
\[
\bf{x}_{\textit{t}} =
\begin{cases} 
(a_i-a_t)(a_j-a_t)~~&~~t\neq i, j, \\ 
-\prod_{r\neq i, j}(a_i-a_r)~~&~~ t=i, \\ 
-\prod_{r\neq i, j}(a_j-a_r)~~&~~t=j.
\end{cases}
\]
From this formula, we see that $\delta(D_{ij})$ is unramified at $w$. In particular, the cocycle $\delta(D_{ij})$ factors through $\textup{Gal}(K_\mathfrak{m}/K)$. Moreover, since $\textup{val}_w(a_i-a_j) = 1$, for all $t\neq i, j$ we have $a_i-a_t\equiv a_j-a_t \mod\mathfrak{p}_w$. Thus, for $t\neq i, j$, the quantity $(a_i-a_t)(a_j-a_t)$ is a square locally at $w$. We conclude that 
\[
\res_w\delta(D_{ij})=(\res_{w}\chi_{-\prod_{r\neq i, j}(a_i-a_r)})\cup D_{ij}.
\] 
Since $\textup{Frob}_{\mathfrak{p}_1}$ agrees with $\textup{Frob}_{\mathfrak{p}_w}$ in the maximal multiquadratic subextension of $K_\mathfrak{m}/K$, we deduce that $\res_{\mathfrak{p}_1}\delta(D_{ij})=(\res_{\mathfrak{p}_1}\chi_{-\prod_{r\neq i, j}(a_i-a_r)})\cup D_{ij}$, which completes the proof of the claim. 

It remains to treat the place $w$. We will show that $\res_w(\textbf{z}_w) \cdot \res_w \delta_{a_{\mathfrak{p}_1\mathfrak{p}_2}}(D_i)$ satisfies the Selmer conditions for $J^{a_{\mathfrak{p}_1\mathfrak{p}_2}}$ at $w$. Using Lemma \ref{descent_lemma}, we see that $\textbf{z}_w \cdot \delta_{a_{\mathfrak{p}_1\mathfrak{p}_2}}(D_i)$ corresponds to the $2g+1$ tuple $\bf{y}$, whose $t$-coordinate is given by
\[
\bf{y}_{\textit{t}} = 
\begin{cases}
a_{\mathfrak{p}_1\mathfrak{p}_2}(a_i-a_t)~~&~~t\neq i, j, \\
a_{\mathfrak{p}_1\mathfrak{p}_w}(a_i-a_j)\prod_{r\neq i, j}(a_i-a_r)~~&~~ t=i, \\ a_{\mathfrak{p}_1\mathfrak{p}_w}a_{\mathfrak{p}_1\mathfrak{p}_2}(a_i-a_j)~~&~~t=j.
\end{cases}
\]
Each coordinate has even valuation so, by Lemma \ref{lem:selmer_condtions_at_mult_places}(i), $\textbf{z}_w \cdot \delta_{a_{\mathfrak{p}_1\mathfrak{p}_2}}(D_i)$ satisfies the Selmer conditions provided $\bf{y}_{\textit{i}}\bf{y}_{\textit{j}}$ is a square locally at $w$. That this is the case follows from \eqref{restriction_at_w_char}.
\end{proof}

Recall that we write $W = \{w_1, \dots, w_{4g^2 + 2g - 1}\}$ and $W' = \{w_1', \dots, w_{2g}'\}$ for the places guaranteed by the genericity of $J$ (see Definition \ref{def:places_in_aux_twist}). For each $1 \leq s \leq 4g^2 + 2g - 1$, let $i_s<j_s$ be elements of $\{1, \dots, 2g + 1\}$ such that $w_s$ has type $\{i_s, j_s\}$.  

\begin{corollary} 
\label{cor:including_zs}
%Let $\textbf{T}\supseteq T$ be a finite set of places. Then 
There is $\kappa_0 \in K^{\times}$ such that:
\begin{itemize}
\item[(1)] $\chi_{\kappa_0}$ is trivial at all $v\in T -W$, and is unramified over $W$,
\item[(2)] the $2$-Selmer group of $J^{\kappa_0}$ contains elements $\textbf{z}_1, \dots, \textbf{z}_{4g^2 + 2g - 1}$ satisfying \eqref{eZiSpec1}--\eqref{eZiSpec4}. Moreover, for $s \in \{1, \dots, 4g^2 + 2g - 1\}$, each $\textbf{z}_s$ has the form $\psi_s\cup D_{i_{s}, j_{s}}$ for a quadratic character $\psi_s$,
\item[(3)] there is a prime $\mathfrak{p}_0\notin T$ such that $\chi_{\kappa_0}$ is ramified at $\mathfrak{p}_0$ and $\res_{\mathfrak{p}_0}(\textbf{z}_i) = 1$ for $i \in \{1, \dots, 4g^2 + 2g - 1\}$.
\end{itemize}
\end{corollary}

\begin{proof}
To begin, we apply Proposition \ref{finding_selmer_basis} with $w=w_1$ and $T'=T$. This produces a quadratic character $\varphi_1$ which is trivial at $v\in T- \{w_1\}$ and is unramified at $w_1$, and an element $\textbf{z}_{1}=\psi_1\cup D_{i_{1}, j_{1}}\in \textup{Sel}_2(J^{\varphi_1})$ with $\textup{val}_{w_1}(\textbf{z}_1) = E_{i_{1}, j_{1}}$ and $\res_v(\textbf{z}_1) = 1$ for all $v\in T- \{w_1\}$. Let $T_1$ be the union of $T$ and the two additional primes $\mathfrak{p}_{1,1}$ and $\mathfrak{p}_{2,1}$ where $\varphi_1$ ramifies (as in the statement of Proposition \ref{finding_selmer_basis}, $\mathfrak{p}_{1, 1}$ ramifies in $\psi_1$).

We now apply Proposition \ref{finding_selmer_basis} with $w=w_2$ and $T' = T_1$. This gives a quadratic character $\varphi_2$ which is trivial at $v\in T_1-\{w_2\}$ and unramified at $w_2$, and an element $\textbf{z}_2=\psi_2\cup D_{i_{2}, j_{2}}\in \textup{Sel}_2(J^{\varphi_2})$ with $\textup{val}_{w_2}(\textbf{z}_2) = E_{i_{2}, j_{2}}$ and $\res_v(\textbf{z}_2) = 1$ for all $v \in T_1-\{w_2\}$. Let $\mathfrak{p}_{1,2}$ and $\mathfrak{p}_{2,2}$ be the two primes where $\varphi_2$ ramifies (noting that $\mathfrak{p}_{1,2}$ ramifies in $\psi_2$).

\textbf{Claim:} for $i=1,2$, we have $\res_{\mathfrak{p}_{i,2}}(\varphi_1) = 1 = \res_{\mathfrak{p}_{i,2}}(\psi_1)$.

Indeed, by reciprocity and properties of the characters $\psi_1$ and $\psi_2$, we have
\[
1 = \prod_v (\psi_1, \psi_2)_v = (\psi_1, \psi_2)_{\mathfrak{p}_{1,2}}.
\] 
It follows that $\res_{\mathfrak{p}_{1,2}}(\psi_1) = 1$. Given this, a similar argument applying reciprocity to the symbol $(\psi_1, \varphi_2)$ shows that $\res_{\mathfrak{p}_{2,2}}(\psi_1) = 1$. Similarly applying reciprocity first to $(\varphi_1, \psi_2)$ and then to $(\varphi_1, \varphi_2)$ completes the proof of the claim. 

We next prove that both $\textbf{z}_1$ and $\textbf{z}_2$ are in the $2$-Selmer group of $J^{\varphi_1 \cdot \varphi_2}.$ Indeed, $\varphi_2$ is trivial at all $v\in T_1- \{w_2\}$, so $\textbf{z}_1$ satisfies the Selmer conditions at all such places (since it does so for the twist $J^{\varphi_1}$). At the places $w_2$, $\mathfrak{p}_{1,2}$ and $\mathfrak{p}_{2,2}$, the character $\psi_1$ is trivial, hence so is $\textbf{z}_1$. At all remaining places, $\textbf{z}_1$ is unramified and $J^{\varphi_1 \cdot \varphi_2}$ has good reduction. 

To see that $\textbf{z}_2$ lies in the $2$-Selmer group of $J^{\varphi_1 \cdot \varphi_2},$ we argue similarly. Using the claim we have that $\varphi_1$ is a square locally at $w_2$, $\mathfrak{p}_{1,2}$ and $\mathfrak{p}_{2,2}$. Thus $\textbf{z}_2$ satisfies the Selmer conditions at these three places (since it does so for the twist $J^{\varphi_2}$). For the remaining places we either have that $\textbf{z}_2$ is trivial locally, or (at places outside $T_1 \cup \{\mathfrak{p}_{1,2},\mathfrak{p}_{2,2}\}$) that $\textbf{z}_2$ is unramified while $J^{\varphi_1 \cdot \varphi_2}$ has good reduction. 

We now set $T_2 = T_1 \cup \{\mathfrak{p}_{1,2},\mathfrak{p}_{2,2}\}$, and continue similarly, using Proposition \ref{finding_selmer_basis} to inductively define characters $\varphi_i$, and associated $1$-cocycles $\textbf{z}_i$, for $i = 1, \dots, 4g^2 + 2g - 1$. We then set $\kappa_0\in K^{\times}$ to be any element whose associated quadratic character is $\prod_{i = 1}^{4g^2 + 2g - 1} \varphi_i$. Arguing as in the case $i = 2$ above, we see that this satisfies the conditions (1) and (2) of the statement. Condition $(3)$ is satisfied by taking $\mathfrak{p}_0$ to be the prime $\mathfrak{p}_{2, 4g^2 + 2g - 1}$ (in the obvious notation). 
\end{proof}

Fix $\kappa_0$  as in Corollary \ref{cor:including_zs}. Let $\mathbf{T}$ be the union of $T$ and all primes where $\chi_{\kappa_0}$ ramifies. For $s\in \{1, \dots, 4g^2 + 2g - 1\}$, we thus have elements $\textbf{z}_s$ of the form $\textbf{z}_s=\psi_s\cup D_{i_{s}, j_{s}}$ as in the statement of Corollary \ref{cor:including_zs}. By construction, the characters $\psi_s$ are linearly independent. Write $K_s$ for the fixed field of the kernel of $\psi_s$. Let $F$ be the compositum of $K_s$ for all $s$ and write $E=K(J[4])F$. Note that $E/K$ is a multiquadratic extension. We extend $\{\psi_1, \dots, \psi_{4g^2 + 2g - 1}\}$ to a basis for $\textup{Hom}(\textup{Gal}(E/K), \mu_2)$ by adding in characters $\eta_1, \dots, \eta_m.$ For each $r \in \{ 1, \dots, m\}$, fix a prime $\mathfrak{p}_r \notin \mathbf{T}$ that splits completely in $F/K$ and is such that $\res_{\mathfrak{p}_r}(\eta_r)$ is non-trivial, yet $\res_{\mathfrak{p}_r}(\eta_{s})$ is trivial for all $s \neq r$.
%Recall that we have a set of $2g$ places $W'=\{w_1', \dots, w_{2g}'\}$, such that $w_i'$ is multiplicative of type $\{i,2g + 1\}$. 
Since, for all $s\in \{1, \dots, 4g^2 + 2g - 1\}$, and all $w'\in W'$, $\psi_s$ has trivial restriction to $w'$, we see that $F/K$ is unramified over $W'$.

\begin{mydef}
We call an element $\lambda \in K^\times/K^{\times 2}$ \textit{promising} if: $\lambda$ has even valuation over $W'$, is trivial at all $v \in \mathbf{T}- W'$, has odd valuation at $\mathfrak{p}_1, \dots, \mathfrak{p}_m$, and is such that if $\lambda$ has odd valuation at a prime $\mathfrak{p}\notin \mathbf{T}$, then $\mathfrak{p}$ splits completely in $F/K$. 
\end{mydef}

For an element $\lambda$, denote by $\textup{Sel}_2^E(J^\lambda)$ the subgroup of $\textup{Sel}_2(J^\lambda)$ consisting of elements with trivial restriction to $E$. 

\begin{remark} 
\label{rem:lin_disjoint}
Let $\lambda$ be promising. By construction, $\kappa_0$ has odd valuation at a prime $\mathfrak{p}_0 \notin T$ such that $\res_{\mathfrak{p}_0}(\textbf{z}_s) = 1$ for all $s\in \{1, \dots, 4g^2 + 2g - 1\}$. This prime is unramified in $E/K$. Since $\lambda$ is promising, we have $\res_{\mathfrak{p}_0}(\lambda) = 1$, hence $\kappa_0 \lambda$ has odd valuation at $\mathfrak{p}_0$ also. In particular, the extensions $E/K$ and $K(\sqrt{\kappa_0\lambda})/K$ are linearly disjoint.
\end{remark}

\begin{lemma} 
\label{lem:systematic_subspace}
Suppose that $\lambda$ is promising. 
\begin{itemize}
 \item[\textup{(1)}] Then $\textup{Sel}_2^{E}(J^{\kappa_0 \lambda})$ is generated by $\textbf{z}_1, \dots, \textbf{z}_{4g^2 + 2g - 1}$.
 \item[\textup{(2)}] Suppose that every element of $\textup{Sel}_2(J^{\kappa_0 \lambda})$ that is unramified over $W'$ is contained in $\textup{Sel}_2^{E}(J^{\kappa_0 \lambda})$. Then $\textup{Sel}_2(J^{\kappa_0 \lambda})$ is generated by $\delta_{\kappa_0 \lambda}(J[2])$ and $\textbf{z}_1, \dots, \textbf{z}_{4g^2 + 2g - 1}$.
\end{itemize}
\end{lemma}

\begin{proof}
(1). First note that, for each $i$, we have $\textbf{z}_i \in \textup{Sel}_2^{E}(J^{\kappa_0 \lambda})$. This follows immediately from the conditions on $\lambda$ and the fact that each $\textbf{z}_i$ is in $\textup{Sel}_2^{E}(J^{\kappa_0})$. (Note that for all $i$ and all $v\in W'$, we have $\res_v(\textbf{z}_i) = 1$.)

Next, take any $\alpha \in \textup{Sel}_2^{E}(J^{\kappa_0 \lambda})$ and represent it as a $2g + 1$ tuple $(\alpha_1, \dots,\alpha_{2g + 1})$ of elements of $K^{\times}/K^{\times 2}$. By assumption, each of the corresponding quadratic characters $\chi_{\alpha_i}$ factor through $\Gal(E/K)$. Fix $k \in \{1, \dots, 2g + 1\}$. Then we can write 
\[
\chi_{\alpha_k} = \prod_{l=1}^{4g^2 + 2g - 1} \psi_l^{\epsilon_{k,l}} \cdot \prod_{r = 1}^m \eta_r^{\epsilon_{k,r}'},
\]
for some $\epsilon_{k,l}, \epsilon_{k,r}'\in \{0,1\}.$ Take $s\in \{1, \dots, m\}$. Since $\chi_{\kappa_0 \lambda}$ is ramified at $\mathfrak{p}_s$, since $\mathfrak{p}_s \notin T$, and since $\alpha$ is unramified at $\mathfrak{p}_s$, the condition that $\alpha \in \textup{Sel}_2(J^{\kappa_0 \lambda})$ forces $\res_{\mathfrak{p}_s}(\alpha) = 1$ (cf.~Lemma \ref{lema:Selmer_conditions}). In particular, we have $\res_{\mathfrak{p}_s}(\chi_{\alpha_k}) = 1$. Now each character $\psi_l$ is trivial locally at $\mathfrak{p}_s$, as is $\eta_r$ for all $r\neq s$. On the other hand, $\res_{\mathfrak{p}_s}(\eta_s)$ is non-trivial. We conclude that $\epsilon_{k,s}' = 0$. Since $s$ and $k$ was chosen arbitrarily, we have $\epsilon_{k,r}' = 0$ for all $k,r$. 

Next, fix $s\in \{1, \dots, 4g^2 + 2g - 1\}$. Then we have 
\[
\res_{w_s}(\chi_{\alpha_k}) = \prod_{l=1}^{4g^2 + 2g - 1} \res_{w_s}(\psi_{l}^{\epsilon_{k,l}}) = \res_{w_s}(\psi_{s}^{\epsilon_{k,s}}).
\]
Now the character $\psi_s$ is ramified at $w_s$. Since $\alpha$ satisfies the Selmer conditions for $J^{\kappa_0}$ at $w_s$ (since $\lambda$ is trivial at $w_s$), we see from Lemma \ref{lem:selmer_condtions_at_mult_places}, Remark \ref{rem:unramf_twist_of_mult}, and the fact that $\kappa_0$ has even valuation at $w_s$, that we have $\epsilon_{k,s} = 0$ unless $k\in \{i_{s}, j_{s}\}$. Moreover, we see that $\epsilon_{i_s,s} = \epsilon_{j_s,s}$. Since this conclusion holds for all $s$, we have
\[
\alpha = \prod_{s=1}^{4g^2 + 2g - 1} \psi_{s}^{\epsilon_{i_s,s}} \cup D_{i_s, j_s} = \prod_{s=1}^{4g^2 + 2g - 1} \textbf{z}_{s}^{\epsilon_{i_s,s}},
\]
completing the proof.

(2). Take any $\alpha \in \textup{Sel}_2(J^{\kappa_0 \lambda})$. By Lemma \ref{lem:selmer_condtions_at_mult_places}, given any $i\in \{1, \dots, 2g\}$, precisely one of $\alpha$ and $\alpha \cdot \delta_{\kappa_0 \lambda}(D_i)$ is unramified at $w_i'$. Moreover, by Lemma \ref{descent_lemma}, for all $j \in \{1, \dots, 2g\}$ with $j \neq i$, we have that $\delta_{\kappa_0 \lambda}(D_j)$ is unramified at $w'_i$. From this we see that some element of the coset $\alpha \cdot \delta_{\kappa_0 \lambda}(J[2])$ is unramified over $W'$. This element lies in $\textup{Sel}_2^{E}(J^{\kappa_0 \lambda})$ by assumption. We now conclude from part (1).
\end{proof}

\begin{lemma} 
\label{lem:nice_chars_shrink}
Suppose that $\lambda$ is promising and that there exists an element $\alpha \in \textup{Sel}_2(J^{\kappa_0 \lambda})$ such that $\alpha$ is unramified over $W'$ and $\alpha \notin \textup{Sel}^E_2(J^{\kappa_0 \lambda})$. Then there is a promising element $\lambda'$ such that $\dim \textup{Sel}_2(J^{\kappa_0 \lambda'}) < \dim \textup{Sel}_2(J^{\kappa_0 \lambda})$.
\end{lemma}

\begin{proof}
We closely follow the proof of \cite{Har19}*{Proposition 4.10}. Let $\lambda$ and $\alpha$ be as in the statement of the lemma. 
%Write $\alpha$ as a $2g + 1$-tuple $(\alpha_1,\dots,\alpha_{2g + 1})$ of elements of $K^{\times}/K^{\times 2}$.
Since $\alpha$ does not factor through $E/K$, one of the $\{\pm1\}$-valued quadratic characters $e(\alpha,D_{1,2g + 1}), \dots, e(\alpha, D_{2g,2g + 1})$ does not factor through $E/K$ (indeed, the points $D_{1, 2g + 1}, \dots, D_{2g, 2g + 1}$ form a basis for $J[2]$, so these characters determine $\alpha$). Choose $i \in \{1, \dots, 2g\}$ such that $e(\alpha, D_{i, 2g + 1})$ does not factor through $E/K$. Writing $\alpha$ as a tuple $(\alpha_1, \dots, \alpha_{2g + 1})$ of elements of $K^{\times}/K^{\times 2},$ the fixed field of the kernel of $e(\alpha, D_{i,2g + 1})$ is then equal to $K(\sqrt{\alpha_i\alpha_{2g + 1}})$.

Denote by $\textup{Ram}(\lambda)$ the set of places where $\lambda$ has odd valuation. Denote by $S$ a finite set of places containing $\mathbf{T}\cup \textup{Ram}(\lambda)$, chosen so that $S- W'$ contains a set of generators for the class group of $K$. This final assumption means that we can find a quadratic extension $L_i/K$ which is ramified at $w_i'$ and unramified outside $S$. 

Let $\mathfrak{q}_0 \notin S$ be a prime that splits completely in $E/K$, and is inert in both $K(\sqrt{\alpha_i\alpha_{2g + 1}})/K$ and $K(\sqrt{\kappa_0\lambda})/K$. The existence of such a prime is guaranteed by Remark \ref{rem:lin_disjoint} and the assumption that $e(\alpha, D_{i,2g + 1})$ does not factor through $E/K$. 

Let $\mathfrak{m}$ be the formal product of $8$ and all places $v \in S- \{w_i'\}$, and denote by $K_\mathfrak{m}/K$ the ray class field of conductor $\mathfrak{m}$. Choose a prime $\mathfrak{q}_1\notin S\cup \{\mathfrak{q}_0\}$ such that:
\begin{itemize}
 \item $\textup{Frob}_{\mathfrak{q}_1} = \textup{Frob}_{\mathfrak{q}_0}^{- 1}$ in $\textup{Gal}(FK_\mathfrak{m}/K)$,
 \item $\textup{Frob}_{\mathfrak{q}_1}\cdot \textup{Frob}_{\mathfrak{q}_0}$ is the non-trivial element of $\textup{Gal}(L_{i}/K)$.
\end{itemize} 
Such a prime exists since $FK_\mathfrak{m}/K$ is unramified at $w_i'$ while $L_{i}/K$ is ramified at $w_i'$.

By construction, the ideal $\mathfrak{q}_0\mathfrak{q}_1$ is principal, with a generator $b$ that is a square locally at all $v\in S- \{w_i'\}$. Thus $\chi_b$ is trivial at all $v \in S-\{w_i'\}$, and is ramified precisely at $\{\mathfrak{q}_0,\mathfrak{q}_1\}.$ Moreover, since $\mathfrak{q}_0$ splits completely in $F/K$, the same is true of $\mathfrak{q}_1$. Consequently, the element $\lambda' := \lambda b$ is promising. We will show that it satisfies the requirements of the statement. 
 
Write $L_i=K(\sqrt{\beta})$. By reciprocity we have 
\begin{eqnarray*}
1&=& \prod_{v \textup{ place of }K} (b, \beta)_v \\ 
&=& (b, \beta)_{w_i'} \cdot (b, \beta)_{\mathfrak{q}_0} \cdot (b, \beta)_{\mathfrak{q}_1} \\ &=&\chi_{b}(\textup{Frob}_{w_i'}) \cdot \chi_{\beta}(\textup{Frob}_{\mathfrak{q}_0}\textup{Frob}_{\mathfrak{q}_1}).
\end{eqnarray*}
Since $\textup{Frob}_{\mathfrak{q}_0}\textup{Frob}_{\mathfrak{q}_1}$ is non-trivial in $\textup{Gal}(L_{i}/K)$ by construction, we deduce that the restriction of $\chi_{b}$ to $w_i'$ is the unique non-trivial unramified character. 

For any $D \in J[2]$, we have 
\begin{equation} 
\label{local_cup_computation}
\res_{\mathfrak{q}_0} \delta_{\kappa_0 \lambda}(D) = \res_{\mathfrak{q}_0} \delta(D) \cdot \left(\res_{\mathfrak{q}_0}(\chi_{\kappa_0 \lambda}) \cup D\right) = \res_{\mathfrak{q}_0}(\chi_{\kappa_0 \lambda}) \cup D, 
\end{equation}
where the first equality is Lemma \ref{changes_under_twist_lemma}, and the second equality follows from the fact that $\mathfrak{q}_0$ splits completely in $K(J[4])/K$. By assumption, $\res_{\mathfrak{q}_0}(\chi_{\kappa_0 \lambda})$ is the unique non-trivial unramified quadratic character, hence $\res_{\mathfrak{q}_0} \delta_{\kappa_0 \lambda}(J[2])$ has dimension $2g$. In particular, we see that $\res_{\mathfrak{q}_0} \delta_{\kappa_0 \lambda}(J[2])$ generates $H^1_{\textup{ur}}(G_{K_{\mathfrak{q}_0}}, J[2])$. 

Now $\alpha$, as an element of $\textup{Sel}_2(J^{\kappa_0 \lambda})$, is unramified at $\mathfrak{q}_0$. Let $D \in J[2]$ be the unique element such that $\res_{\mathfrak{q}_0}(\alpha) = \res_{\mathfrak{q}_0}\delta_{\kappa_0 \lambda}(D)$. By construction, $\res_{\mathfrak{q}_0} e(\alpha,D_{i,2g + 1})$ is non-trivial, hence by \eqref{local_cup_computation} we have $e(D,D_{i,2g + 1}) =- 1$. In particular, when we write $D$ in the basis $D_1, \dots, D_{2g}$, the coefficient of $D_i$ is non-zero. It follows from Lemma \ref{lem:selmer_condtions_at_mult_places} that $\res_{w_i'}\delta_{\kappa_0 \lambda}(D)$ has non-trivial image in $\delta(J^{\kappa_0\lambda}(K_{w_i'}))/\delta(J^{\kappa_0\lambda}(K_{w_i'})) \cap H^1_{\textup{ur}}(G_{K_{w_i'}}, J[2])$,  while $\res_{w_i'}(\alpha)$ has trivial image in this quotient (by assumption, $\alpha$ is unramified at $w_i'$). Consequently, the image under restriction of $\left \langle \alpha, \delta_{\kappa_0 \lambda}(J[2]) \right \rangle$ in 
\[
H^1(G_{K_{\mathfrak{q}_0}}, J[2]) \oplus \frac{\delta(J^{\kappa_0\lambda}(K_{w_i'}))}{\delta(J^{\kappa_0\lambda}(K_{w_i'})) \cap H^1_{\textup{ur}}(G_{K_{w_i'}}, J[2])}
\]
is $2g + 1$ dimensional. We now apply Lemma \ref{lem:mazur_rubin_harpaz} with $X = \{w_i', \mathfrak{q}_0, \mathfrak{q}_1\}$. Combining the above discussion with Lemma \ref{lem:mazur_rubin_harpaz} and Lemma \ref{lHarpazChange}, we conclude that $\dim \textup{Sel}_2(J^{\kappa_0 \lambda'}) < \dim \textup{Sel}_2(J^{\kappa_0 \lambda})$, as desired.
\end{proof}

\begin{proof}[Proof of Theorem \ref{thm:almost_aux_twist}]
We first construct a promising element. Let $\mathfrak{m}$ be the formal product of $8$ and all places $v\in \textbf{T}$. Let $K_\mathfrak{m}$ denote the corresponding ray class field. With the primes $\mathfrak{p}_1, \dots, \mathfrak{p}_m$ as above, choose a prime $\mathfrak{p}_{m + 1} \notin \textbf{T}\cup \{\mathfrak{p}_1, \dots, \mathfrak{p}_{m}\}$ such that 
\[
\textup{Frob}_{\mathfrak{p}_{m + 1}} = \prod_{i=1}^{m}\textup{Frob}_{\mathfrak{p}_i}^{-1}
\]
in $\textup{Gal}(K_\mathfrak{m}F/K)$. Then $\mathfrak{p}_{m + 1}$ splits completely in $F/K$. Moreover, the ideal $\prod_{i=1}^{m + 1}\mathfrak{p}_i$ is principal, and has a generator $\lambda \in K^{\times}$ which is a square locally at all places $v\in \textbf{T}$. The element $\lambda$ is then promising. 

Inductively applying Lemma \ref{lem:nice_chars_shrink}, starting from the element $\lambda$ just constructed, we deduce the existence of a promising element $\lambda'$ such that every element $\alpha \in \textup{Sel}_2(J^{\kappa_0 \lambda'})$ that is unramified over $W'$ lies in $\textup{Sel}_2^E(J^{\kappa_0 \lambda'})$. By Lemma \ref{lem:systematic_subspace}, $\textup{Sel}_2(J^{\kappa_0 \lambda'})$ is generated by $\delta_{\kappa_0 \lambda'}(J[2])$ and $\textbf{z}_1, \dots, \textbf{z}_{4g^2 + 2g - 1}$. Moreover, as in Remark \ref{rem:lin_disjoint}, $\chi_{\kappa_0 \lambda'}$ is ramified at the prime $\mathfrak{p}_0$ afforded by Corollary \ref{cor:including_zs}, and all the $\textbf{z}_i$ have trivial restriction at this prime. The character $\chi := \chi_{\kappa_0 \lambda'}$ satisfies the requirements of the statement.
\end{proof}

\bibliography{bibliography}
\end{document}